\documentclass[final]{ws-m3as}
\usepackage{lipsum}
\usepackage{amsfonts}
\usepackage{graphicx}
\usepackage{epstopdf}
\usepackage{algorithmic}
\usepackage{verbatim}
\usepackage{tikz}
\usepackage{pgfplots}
\usepackage{bm}
\definecolor{mygreen}{HTML}{43a047}
\definecolor{hone}{HTML}{d50000}
\definecolor{htwo}{HTML}{283593}
\newcommand{\ul}[1]{\underline{#1}}
\def\usmall#1{#1}
\usepackage{mathtools}
\begin{document}
\markboth{Barbara Kaltenbacher, Vanja Nikoli\'c}{On the Jordan--Moore--Gibson--Thompson equation}
\title{ON THE JORDAN--MOORE--GIBSON--THOMPSON EQUATION:\\
WELL-POSEDNESS WITH QUADRATIC GRADIENT NONLINEARITY AND SINGULAR LIMIT FOR VANISHING RELAXATION TIME
}
\author{BARBARA KALTENBACHER}
\address{Alpen-Adria-Universit\"at Klagenfurt, Institut f\"ur Mathematik, \\ Universit\"atsstra\ss e 65–67, 9020~Klagenfurt, Austria. \\
barbara.kaltenbacher@aau.at}
\author{VANJA NIKOLI\'C}
\address{Technical University of Munich,  Department of Mathematics, \\
	Boltzmannstra\ss e 3, 85748 Garching, Germany.\\
vanja.nikolic@ma.tum.de}
\maketitle

\begin{abstract}
In this paper, we consider the Jordan--Moore--Gibson--Thompson equation, a third order in time wave equation describing the nonlinear propagation of sound that avoids the infinite signal speed paradox of classical second order in time strongly damped models of nonlinear acoustics, such as the Westervelt and the Kuznetsov equation. We show well-posedness in an acoustic velocity potential formulation with and without gradient nonlinearity, corresponding to the Kuznetsov and the Westervelt nonlinearities, respectively.
Moreover, we consider the limit as the parameter of the third order time derivative that plays the role of a relaxation time tends to zero, which again leads to the classical Kuznetsov and Westervelt models.
To this end, we establish appropriate energy estimates for the linearized equations and employ fixed-point arguments for well-posedness of the nonlinear equations.  The theoretical results are illustrated by numerical experiments.
\end{abstract}



\section{Introduction}
Nonlinear propagation of sound arises in numerous applications. We here especially mention high-intensity ultrasound used in medical imaging and therapy, but also for industrial purposes, such as ultrasound cleaning or welding; see, e.g., Refs.~\refcite{abramov,dreyer2000,uffc2002}, and the given references therein.
For the physical fundamentals of nonlinear acoustics, we refer to, e.g., Refs.~\refcite{Crighton79,EnfloHedberg,HamiltonBlackstock98,Kuznetsov71,MakarovOchmann96,MakarovOchmann97b,MakarovOchmann97a,Pierce89}.

Its physical and mathematical description involves the acoustic particle velocity $\vec{v}$, the acoustic pressure $p$, as well as the mass density $\varrho$, which can be decomposed into constant and fluctuating components
\[
\vec{v}=\vec{v}_0+\vec{v}_\sim\,, \quad p=p_0+p_\sim\,, \quad \varrho=\varrho_0+\varrho_\sim,
\]
where in the applications mentioned above, the ambient flow vanishes; i.e, $\vec{v}_0=0$. Furthermore, we have the balances of mass, momentum and sometimes of energy, complemented with an equation of state that relates the mass density to the pressure.  
Combination of these balance and material laws yields wave-type partial differential equations that are often second order in space and time, but higher order in time equations  play an important role as well. It is one of these third order in time equations that we focus on in this paper.   

One of the most established models of nonlinear acoustics is Kuznetsov's equation\cite{Kuznetsov71,LesserSeebass68}
\begin{equation}\label{Kuznetsov0}
{p_\sim}_{tt}- c^2\Delta p_\sim - \delta \Delta {p_\sim}_t =
\left(\frac{1}{\varrho_0 c^2}\frac{B}{2A} p_\sim^2 + \varrho_0 |\vec{v}|^2\right)_{tt}\,,
\end{equation}
where $c$ is the speed of sound, $\delta$ is the diffusivity of sound
\[
\delta = \frac{1}{\varrho_0}\left( \frac{4 \mu_V}{3} + \zeta_V \Big) + \frac{\kappa}{\varrho_0}
\Big( \frac{1}{c_V} - \frac{1}{c_p} \right),
\]
and the velocity is related to the pressure via some balance of forces,
\begin{equation}\label{eq:forcebal}
\varrho_0\vec{v}_t=-\nabla p\,.
\end{equation}
By ignoring local nonlinear effects modeled by the quadratic velocity term, we arrive at the Westervelt equation\cite{Westervelt63}
\begin{equation}\label{Westervelt0}
{p_\sim}_{tt}- c^2\Delta p_\sim - \delta \Delta {p_\sim}_t =
\frac{\beta_a}{\varrho_0 c^2} {p_\sim^2}_{tt},
\end{equation}
with $\beta_a = 1 + B/(2A)$. In terms of the acoustic velocity potential $\psi$ satisfying $\vec{v}=-\nabla\psi$ and $p=\varrho_0\psi_t$, these equations can be rewritten as 
\begin{equation}\label{Kuznetsov}
\psi_{tt}-c^2\Delta \psi - \delta\Delta \psi_t
= \left(\frac{1}{c^2}\frac{B}{2A}(\psi_t)^2+|\nabla\psi|^2\right)_t\,,
\end{equation}
and 
\begin{equation}\label{Westervelt}
\psi_{tt}-c^2\Delta \psi - \delta\Delta \psi_t
= \left(\frac{\beta_a}{c^2}(\psi_t)^2\right)_t\,,
\end{equation}
respectively.

As has been observed in, e.g., Ref.~\refcite{JordanMaxwellCattaneo14}, the use of classical Fourier's law 
\[
\mathbf{q}=-K\nabla\vartheta
\]
where $\vartheta$, $\mathbf{q}$, and $K$ denote the absolute temperature, heat flux vector, and thermal conductivity, respectively, leads to an infinite signal speed paradox, which appears to be unnatural in wave propagation. Therefore in Ref.~\refcite{JordanMaxwellCattaneo14}, several other constitutive relations for the heat flux are considered within the derivation of nonlinear acoustic wave equations. Among these is the Maxwell--Cattaneo law
\[
\tau\mathbf{q}_t+\mathbf{q}=-K\nabla\vartheta,
\]
where $\tau$ is a positive constant accounting for relaxation (the relaxation time), whose combination with the above mentioned balance equations and the equation of state leads to the third order in time PDE model:
\begin{equation}\label{KuznetsovMC}
\tau\psi_{ttt}+\psi_{tt}-c^2\Delta \psi - b\Delta \psi_t
= \left(\frac{1}{c^2}\frac{B}{2A}(\psi_t)^2+|\nabla\psi|^2\right)_t\,,
\end{equation}
where \begin{equation}\label{btau}
b=\delta+\tau c^2\,.
\end{equation}
This model is known in the literature as the Jordan--Moore--Gibson--Thompson (JMGT) equation\cite{KLP12_JordanMooreGibson} and we refer to Ref.~\refcite{JordanMaxwellCattaneo14} for its derivation. If one neglects local nonlinear effects modeled by the quadratic velocity term in \eqref{KuznetsovMC}, one arrives at
\begin{equation}\label{WesterveltMC}
\tau\psi_{ttt}+\psi_{tt}-c^2\Delta \psi - b\Delta \psi_t
= \left(\frac{\beta_a}{c^2}(\psi_t)^2\right)_t,
\end{equation}
analogously to the reduction of the Kuznetsov to the Westervelt equation\cite{KLP12_JordanMooreGibson}. We will refer to equation \eqref{KuznetsovMC} as the Kuznetsov-type and to \eqref{WesterveltMC} as the Westervelt-type JMGT equation. Obviously, equations \eqref{KuznetsovMC} and \eqref{WesterveltMC} formally reduce to \eqref{Kuznetsov} and \eqref{Westervelt} upon setting $\tau=0$. The present work is, in part, devoted to the rigorous justification of passing to the limit $\tau\to0$ in \eqref{KuznetsovMC} and \eqref{WesterveltMC}. \\
\indent In Ref.~\refcite{KLM12_MooreGibson} and, more comprehensively, in Ref.~\refcite{MarchandMcDevittTriggiani12}, the linearized equation
\begin{equation}\label{WesterveltMC_lin}
\tau\psi_{ttt}+\alpha\psi_{tt}-c^2\Delta \psi - b\Delta \psi_t
= f\,,
\end{equation}
often called the Moore--Gibson--Thompson equation, is studied using semigroup techniques; see also Refs.~\refcite{DellOroPata} and~\refcite{PellicerSolaMorales}. As it turns out, the exponential stability of the trajectories depends on the critical parameter given by
\begin{equation}\label{eq:gamma}
\gamma := \alpha  - \frac{\tau c^2}{b} \,.
\end{equation}
In the case of a constant coefficient $\alpha$, exponential decay of the energy function
\begin{equation}\label{eq:EWK}
E[\psi](t) = \tfrac12\left\{|\psi_{t}(t)|^2 +|\nabla \psi(t)|^2
+|\psi_{tt}(t)|^2 +|\nabla \psi_t(t)|^2 + |\Delta \psi(t)|^2 \right\}
\end{equation} 
requires $\gamma$ to be strictly positive. The case $\gamma<0$ is unstable and the case $\gamma=0$ marginally stable.
An intuitive explanation for this phenomenon is the following: According to the linear wave part of the equation, we can trade $\alpha \psi_{tt}$ for $c^2\Delta \psi$, thus also $\tau \psi_{ttt}$ for $\frac{\tau c^2}{\alpha}\Delta \psi_t$ in order to relate \eqref{WesterveltMC_lin} back to the linearization of \eqref{Westervelt}
\[
\alpha\psi_{tt}-c^2\Delta \psi - \frac{b}{\alpha}\gamma \Delta \psi_t
= f\,,
\]
which is a strongly damped wave equation. \\
\indent We mention that the Moore--Gibson--Thompson equation \eqref{WesterveltMC_lin} is also studied in Ref.~\refcite{LiuTriggiani13}, where the problem of identifying $\gamma(t)$ from boundary measurements is considered, and in Ref.~\refcite{LasieckaWang15a} and~\refcite{LasieckaWang15b}, where the effect of additive convolution memory terms acting on $\Delta u$ and $\Delta u_t$, and their combination, respectively, is investigated. \\
\indent For a reformulation of the nonlinear equation \eqref{WesterveltMC} in terms of the acoustic pressure $p=\varrho_0\psi_t$, global in time well-posedness and exponential decay of the energy $E[p](t)$ is proven in Ref.~\refcite{KLP12_JordanMooreGibson} for small initial data $(p_0,p_1,p_2)\in H_0^1(\Omega)\cap H^2(\Omega)\times H_0^1(\Omega)\times L^2(\Omega)$, where $\Omega$ is a bounded $C^2$ smooth domain.\\
\indent One of the key elements in the above cited works on the analysis of equations \eqref{WesterveltMC} and \eqref{WesterveltMC_lin} is introduction of the auxiliary state
\begin{equation}\label{eq:z}
z:= \psi_t + \frac{c^2}{b}\psi \,.
\end{equation}
Indeed, in this manner, the third order in time equation \eqref{WesterveltMC_lin} is reduced to a linear (weakly) damped wave equation for $z$, 
\begin{equation}\label{dampedwave_z}
\tau z_{tt} +\gamma z_t - b \Delta z-\gamma\tfrac{c^2}{b} z+\gamma\tfrac{c^4}{b^2} \psi=f\,,
\end{equation}
where $-\gamma\tfrac{c^2}{b} z+\gamma\frac{c^4}{b^2} \psi$ is a lower order term. \\
\indent This approach, first of all, illustrates the fact that $\gamma$ should be non-negative to guarantee a damping behaviour of the term $\gamma z_t$. Secondly, it displays the key difference to the strongly damped second order equations \eqref{Kuznetsov} and \eqref{Westervelt}. As pointed out in Subsection 6.2.1 of Ref.~\refcite{MarchandMcDevittTriggiani12}, equation \eqref{WesterveltMC_lin} does not give rise to an analytic semigroup; see also Remark 1.3 in Ref.~\refcite{KLM12_MooreGibson}. Consequently, the operator driving the evolution does not exhibit maximal parabolic regularity\cite{LeCroneSimonett} and the Implicit function theorem argument from, e.g., Ref.~\refcite{MW13} cannot be transferred to the present setting.

\section{Main results} \label{sec:main}
This paper contributes to the analysis of the JMGT equation in two ways. \\
\indent Firstly, we prove well-posedness with a quadratic gradient nonlinearity arising when taking into account local nonlinear effects; cf. the additional $(|\nabla\psi|^2)_t$ term on the right hand side in \eqref{KuznetsovMC} compared to \eqref{WesterveltMC}. 
We base our approach on energy estimates for a reformulation of \eqref{KuznetsovMC} in the form
\begin{equation}\label{fp}
\tau\psi_{ttt}+(1-k\psi_t)\psi_{tt}-c^2\Delta \psi - b\Delta \psi_t
= 2\nabla\psi\cdot\nabla\psi_t,
\end{equation}
where we use the abbreviation $k=\frac{2}{c^2}\frac{B}{2A}$. The sign of $k$ will not matter in what follows, whereas we assume the coefficients $b$ and $c^2$ to be strictly positive. We rely on the formulation of the equations in terms of the acoustic velocity potential $\psi$ and not the acoustic pressure\cite{KLP12_JordanMooreGibson}, since it allows to include more easily the quadratic velocity term $(|\nabla\psi|^2)_t$ on the right-hand side. We note that the energy estimates required for this purpose differ from those provided in Ref.~\refcite{KLP12_JordanMooreGibson} for the equation \eqref{WesterveltMC} without the quadratic gradient nonlinearity. \\
\indent Secondly, we consider the limit $\tau\to0$ and prove that solutions of the Kuznetsov-type equation \eqref{KuznetsovMC} converge to a solution of equation \eqref{Kuznetsov} as $\tau\to0$, and analogously for the Westervelt-type equation. For this purpose, the energy estimates we will derive are crucial.
These estimates differ for the Kuznetsov-type and for the Westervelt-type version of the JMGT equation, which is why we treat these models in separate sections. \\
\indent The rest of the paper is organized as follows. In Section~\ref{sec:enest}, we consider the linearized equation \eqref{WesterveltMC_lin} with a fixed positive coefficient $\alpha$ that is possibly space and time dependent, but bounded away from zero, and an inhomogeneity $f$, as well as a fixed positive $\tau$. We prove well-posedness of this linearized model together with an energy estimate. \\
\indent Section \ref{sec:wellposed_Wes} contains a well-posedness proof for the Westervelt version \eqref{WesterveltMC} of the equation by setting $\alpha=1-k\psi_t$ and $f=0$. The proof is based on the equation \eqref{fp}, but with zero right-hand side, as the gradient nonlinearity is not present in \eqref{WesterveltMC}. This fact allows to prove local in time well-posedness for small inital data, even without any sign condition on the parameter $\gamma= \alpha  - \frac{\tau c^2}{b}$. However, the energy estimates from Section~\ref{sec:enest} do not cover the gradient nonlinearity in the Kuznetsov-type version of the JMGT equation, so that higher order energy estimates are needed. We derive them in Section~\ref{sec:enest_higher}. These involve the auxiliary function $z$  and require positivity of both $\alpha$ and $\gamma$, where the latter follows from positivity and boundedness away from zero of $\alpha$ for $\tau$ sufficiently small. \\
\indent Section~\ref{sec:wellposed_Kuz} provides the corresponding well-posedness result for the equation \eqref{KuznetsovMC} based on reformulation \eqref{fp}; i.e., setting $\alpha=1-k\psi_t$ and $f=2\nabla\psi\cdot\nabla\psi_t$.
Starting from a sufficiently small positive value and letting $\tau$ tend to zero clearly preserves the sign structure of the  coefficients, in particular of $\gamma$ and $b$, so that the energy bounds from Sections~\ref{sec:wellposed_Wes} and \ref{sec:wellposed_Kuz} can be used for justifying the limiting process $\tau\to0$ in Section \ref{sec:limits}. There we also provide a brief comparison of the resulting regularity for the limiting equations \eqref{Kuznetsov}, \eqref{Westervelt} to those already present in the literature. Finally, numerical experiments in Section \ref{sec:NumEx} illustrate the theoretical findings.
\begin{remark}[On medium parameters]
We require strict positivity of constants $c$, $\delta$, and $\tau$ (hence, also of $b=\delta+\tau c^2$) appearing in the equations for proving well-posedness of initial-boundary value problems \eqref{KuznetsovMC} and \eqref{WesterveltMC}. These are, indeed, very natural assumptions from a physical point of view; typical values of these parameters in different media can be found, for example, in Refs. \refcite{crocker1998handbook} and~\refcite{Rossing}. It is also known that, in order to establish global well-posedness for the limiting problems \eqref{Kuznetsov} and \eqref{Westervelt} in space dimensions higher than one, strict positivity of $\delta$ is needed\cite{KL09Westervelt,KL12_Kuznetsov}. The constant $k$ does not need to have a particular sign in our analysis, but will typically be non-negative in applications.
\end{remark}
\subsection{Theoretical preliminaries and assumptions}
We set here the notation and collect some useful theoretical results that we often use in the analysis. Throughout the paper, the spatial domain $\Omega \subset \mathbb{R}^d$, where $d \in \{1, 2, 3\}$, is assumed to be sufficiently smooth to admit integration by parts as well as second order elliptic regularity. \\
\indent We consider the PDEs on a bounded space time cylinder $\Omega\times(0,T)$ and impose homogeneous Dirichlet boundary conditions on $\partial\Omega$ for simplicity
\begin{equation}\label{Dirichlet}
\psi=0 \quad \mbox{ on } \partial \Omega\times(0,T);
\end{equation}
i.e., $(-\Delta):H_0^1(\Omega)\to H^{-1}(\Omega)$ is the Laplace operator equipped with homogeneous Dirichlet boundary conditions.
We expect that Neumann and impedance (absorbing) boundary conditions can be treated analogously, but lead to modifications in the energy estimates. \\
\indent The third order in time equations \eqref{KuznetsovMC} and \eqref{WesterveltMC} are complemented with initial conditions
\begin{equation}\label{eq:ini_psi}
\psi(0)=\psi_0 \,, \quad \psi_t(0)=\psi_1\,, \quad \psi_{tt}(0)=\psi_2\,,
\end{equation}
whereas for the limiting second order in time equations \eqref{Kuznetsov}, \eqref{Westervelt} as $\tau\to0$, the initial condition on $\psi_{tt}$ naturally disappears, which will be seen also in the energy estimates.
\subsubsection{Notation}
For simplicity of notation, we often omit the time interval and the spatial domain when writing norms, i.e., $\|\cdot\|_{L^p L^q}$ denotes the norm on $L^p(0,T;L^q(\Omega))$. We also abbreviate the $L^2(\Omega)$ inner product by $(\cdot,\cdot)_{L^2}$ and the $L^2(\Omega)$ norm (as well as the absolute value) by $|\cdot|$. We employ the notation $L^p(0,T;Z)$, $W^{s,p}(0,T;Z)$ for the Bochner-Sobolev spaces of time dependent functions.\\
\indent More specifically, we use dedicated function spaces for the solutions of the considered equations. We collect their notation here for future reference:
\begin{equation}\label{XWXKspaces}
\begin{aligned}
X^W :=& \,\begin{multlined}[t] W^{1, \infty}(0;T;H_0^1(\Omega) \cap H^2(\Omega)) \cap W^{2, \infty}(0,T; H_0^1(\Omega)),\\
	\cap H^3(0,T; L^2(\Omega)) \end{multlined} \\
X^K:=&\, L^\infty(0,T;H^3(\Omega))\cap W^{1,\infty}(0,T;H^2(\Omega))\cap W^{2,\infty}(0;T;H_0^1(\Omega)),\\
\bar{X}^W :=& \, W^{1, \infty}(0;T;H_0^1(\Omega) \cap H^2(\Omega)) \cap H^2(0,T; H_0^1(\Omega))\\
\bar{X}^K:=&\, L^\infty(0,T;H^3(\Omega))\cap W^{1,\infty}(0,T;H^2(\Omega))\cap H^2(0;T;H_0^1(\Omega)),
\end{aligned}
\end{equation}
with partly $\tau$-dependent norms induced by the energy estimates to be derived 
\begin{equation}\label{XWXKnorms}
\begin{aligned}
\|\psi\|_{X^W}^2:=& \,
\|\psi\|^2_{W^{1,\infty} H^2}+\|\psi_{tt}\|^2_{L^2 H^1}+ \tau \|\psi_{tt}\|^2_{L^\infty H^1}+\tau^2 \|\psi_{ttt}\|^2_{L^2L^2}, \\
\|\psi\|_{X^K}^2:=& \, \begin{multlined}[t]
\|\psi\|_{L^\infty H^3}^2+\|\psi_t+\tfrac{c^2}{b}\psi \|_{L^\infty H^2 }^2
    +\|\psi_{tt}+\tfrac{c^2}{b}\psi_t \|_{L^2 H_0^1 }^2 \\
	+\tau \|\psi_{tt}+\tfrac{c^2}{b}\psi_t \|_{L^\infty H_0^1 }^2,\end{multlined} \\
\|\psi\|_{\bar{X}^W}^2:=&\, \|\psi\|^2_{W^{1,\infty} H^2}+\|\psi_{tt}\|^2_{L^2 H^1}, \\
\|\psi\|_{\bar{X}^K}^2:=&\, \|\psi_{tt}+\tfrac{c^2}{b}\psi_t\|^2_{L^2 H^1}+\|\psi_t+\tfrac{c^2}{b}\psi\|^2_{L^\infty H^2}+\|\psi\|^2_{L^\infty H^3}.
\end{aligned}
\end{equation}
\subsubsection{Helpful inequalities} Throughout the paper, we often employ the continuous embeddings
$H^1(\Omega)\hookrightarrow L^6(\Omega)$, $H^2(\Omega) \hookrightarrow L^\infty(\Omega)$
\begin{equation}\label{embeddigs}
\|v\|_{L^6(\Omega)}\leq C_{H^1, L^6}^\Omega \|v\|_{H^1(\Omega)}
\,, \quad 
\|v\|_{L^\infty(\Omega)}\leq C_{H^2, L^\infty}^\Omega \|v\|_{H^2(\Omega)}
\end{equation}
as well as boundedness of the operator $(-\Delta)^{-1}:L^2(\Omega)\to H^2(\Omega)\cap H_0^1(\Omega)$, the Poincar\'{e}-Friedrichs inequality, 
\begin{equation}\label{ellreg-PF}
\|v\|_{H^2(\Omega)}\leq C_{(-\Delta)^{-1}}^\Omega \|-\Delta v\|_{L^2(\Omega)}
\,, \quad 
\|v\|_{H^1(\Omega)}\leq C_{PF}^\Omega \|\nabla v\|_{L^2(\Omega)},
\end{equation}
and the trace theorem
\begin{equation}\label{ellreg-PF_trace}
\|\nu\cdot \nabla v\|_{H^{-1/2}(\partial\Omega)}\leq C_{tr}^\Omega \|v\|_{H^1(\Omega)}
\,, \quad 
\|v\|_{H^{1/2}(\partial\Omega)}\leq C_{tr}^\Omega \|v\|_{H^1(\Omega)},
\end{equation}
see, e.g., Lemma 4.3 in Ref.~\refcite{McLean00}, where $\nu$ denotes the outward unit normal on the boundary of $\Omega$.

\section{Analysis of the linear damped wave equation \eqref{dampedwave_z}} \label{sec:enest}
\indent We now turn our attention to the linear equation \eqref{WesterveltMC_lin} and study the following initial-boundary value problem:
\begin{equation} \label{ibvp_linear}
\begin{aligned}
\begin{cases}
\tau\psi_{ttt}+\alpha(x,t)\psi_{tt}-c^2\Delta \psi - b\Delta \psi_t = f \quad \mbox{ in }\Omega\times(0,T), \\[1mm]
\psi=0 \quad \mbox{ on } \partial \Omega\times(0,T),\\[1mm]
(\psi, \psi_t, \psi_{tt})=(\psi_0, \psi_1, \psi_2) \quad \mbox{ in }\Omega\times \{0\},
\end{cases}
\end{aligned}
\end{equation}
under the non-degeneracy assumption that for some $\ul{\alpha}>0$
\begin{align} \label{non-degeneracy_assumption}
\alpha(t)\geq\ul{\alpha}\ \mbox{ on }\Omega \ \ \mbox{  a.e. in } \Omega \times (0,T).
\end{align}
We assume that the coefficient $\alpha$ and the source term $f$ have the following regularity:
\begin{equation}\label{eq:alphagammaf_reg_Wes}
\begin{aligned} 
&\alpha \in 
L^\infty(0,T; L^\infty(\Omega))\cap L^\infty(0,T; W^{1,3}(\Omega)), \\
&f\in H^1(0,T; L^2(\Omega)).
\end{aligned}
\end{equation}
Moreover, $\psi$ is assumed to satisfy the initial conditions \eqref{eq:ini_psi} with 
\begin{equation}\label{reg_init}
(\psi_0, \psi_1, \psi_2)\in X^W_0 :=
H_0^1(\Omega) \cap H^2(\Omega)\times H_0^1(\Omega) \cap H^2(\Omega)\times H_0^1(\Omega).
\end{equation} 
For an analysis of equation \eqref{WesterveltMC_lin} with constant coefficient $\alpha$, under the slightly weaker assumptions $$(\psi_0, \psi_1, \psi_2)\in (H_0^1(\Omega) \cap H^2(\Omega))\times H_0^1(\Omega) \times L^2(\Omega)$$ and $f\in L^1(0,T;L^2(\Omega))$, we refer to Corollary 1.2 in Ref.~\refcite{KLM12_MooreGibson}. There it is shown that $\psi\in C(0,T;H_0^1(\Omega) \cap H^2(\Omega))\cap \, C^1(0,T;H_0^1(\Omega))\cap \, C^2(0,T;L^2(\Omega))$ by means of semigroup techniques.
The assumptions \eqref{eq:alphagammaf_reg_Wes} and \eqref{reg_init} naturally arise from the energy estimates in the well-posedness proof below and lead to the stronger (as compared to Ref.~\refcite{KLM12_MooreGibson}) regularity stated in \eqref{regularity} below.

\begin{theorem} \label{th:wellposedness_lin}
	Let $c^2$, $b$, $\tau>0$, and let $T>0$ be a fixed time horizon. Let the non-degeneracy assumption \eqref{non-degeneracy_assumption} and the regularity assumptions \eqref{eq:alphagammaf_reg_Wes} and \eqref{reg_init} hold. Then there exists a unique  solution $\psi$ of the problem \eqref{ibvp_linear} that satisfies 
	\begin{equation}\label{regularity}
	\begin{aligned}
	\psi \in \, X^W :=& \, W^{1, \infty}(0;T;H_0^1(\Omega) \cap H^2(\Omega)) \cap W^{2, \infty}(0,T; H_0^1(\Omega))\\
	&\cap H^3(0,T; L^2(\Omega)).
	\end{aligned}
	\end{equation}
Furthermore, the solution fullfils the estimate
	\begin{equation}\label{energy_est_lin}
	\begin{aligned}
&  \tau^2 \|\psi_{ttt}\|^2_{L^2L^2} + \tau \|\psi_{tt}\|^2_{L^\infty H^1}+\|\psi_{tt}\|^2_{L^2 H^1}+\|\psi_t\|^2_{L^\infty H^2}
\\[1mm]
\leq&\,C(\alpha, T, \tau)\left(|\psi_0|^2_{H^2}+|\psi_1|^2_{H^2}+\tau|\psi_2|^2_{H^1}+ \| f\|^2_{L^\infty L^2}+\|f_t\|^2_{L^2 L^2}\right). 
	\end{aligned}
	\end{equation}
The constant above is given by
\begin{equation*}
\begin{aligned}
C(\alpha, T, \tau)= C_1 \, \left(1+T^3+\|\alpha\|_{L^\infty L^\infty}\right)\, \textup{exp}\,\left(C_2 \, \left(\tfrac{1}{\tau}\|\nabla \alpha\|^2_{L^\infty L^3} +1+T \right) T \right),
\end{aligned}
\end{equation*}
where $C_1$, $C_2>0$ do not depend on $\tau, T$, or $\alpha$.

If additionally we assume that $\|\nabla \alpha\|_{L^\infty L^3}$ is sufficiently small so that
\begin{equation}\label{gradalphasmall}
\|\nabla \alpha\|_{L^\infty L^3}<\frac{\underline{\alpha}}{C^\Omega_{H^1,L^6} }
\end{equation}
holds, then \eqref{energy_est_lin} is valid with an upper bound that is independent of $\tau$, i.e., 
\begin{equation}\label{CT}
\begin{aligned}
C(\alpha, T, \tau)= C(\alpha,T) = C_1 \, \left(1+T^3+\|\alpha\|_{L^\infty L^\infty}\right)\, \textup{exp}\,\left(C_2 \, (1+T) T \right).
\end{aligned}
\end{equation}

\end{theorem}
\begin{proof}
We carry out the proof by via Galerkin approximations in space, relying on energy estimates; cf. Refs.~\refcite{EvansBook,Roubicek}. Note that the initial data are meaningful since regularity \eqref{regularity} implies
\begin{equation*}
\begin{aligned}
&\psi \in C([0,T]; H_0^1(\Omega) \cap H^2(\Omega)), \\
&\psi_t \in C_w([0,T]; H_0^1(\Omega) \cap H^2(\Omega)), \\
&\psi_{tt} \in C_w([0,T]; H_0^1(\Omega)),
\end{aligned}
\end{equation*}
where $C_w$ denotes the space of weakly continuous functions; see Lemma 3.3 in Ref.~\refcite{Temam}. 
\\[3mm]
\noindent \textbf{Step 1: Discretization in space.} Let $\{w_i\}_{i \in \mathbb{N}}$ denote the eigenfunctions of the Dirichlet-Laplacian operator $- \Delta$. Then $\{w_i\}_{i \in \mathbb{N}}$ can be normalized to form an orthogonal basis of $H_0^1(\Omega) \cap H^2(\Omega)$ and to be orthonormal with respect to the $L^2(\Omega)$ scalar product.  \\
\indent Fix $n \in \mathbb{N}$ and denote $V_n=\text{span}\{w_1, \ldots, w_n\}$. We seek an approximate solution in the form of
\begin{equation}
\begin{aligned}
\psi^n=& \, \displaystyle \sum_{i=1}^n \xi_i(t)w_i(x),\\
\end{aligned}
\end{equation}
where $\xi_i:(0,T) \rightarrow \mathbb{R}$, $i \in [1,n]$. The initial data are chosen as
\begin{align*}
\psi^n_0(x)=\displaystyle \sum_{i=1}^n \xi_{i, 0}\, w_i(x),  \quad \psi^n_1(x)=\displaystyle \sum_{i=1}^n \xi_{i,1} \, w_i(x), \quad \psi^n_2(x)=\displaystyle \sum_{i=1}^n \xi_{i,2} \, w_i(x),
\end{align*}
where the coefficients $\xi_{i,0}$, $\xi_{i,1}$, $\xi_{i,2} \in \mathbb{R}$ are given by
\begin{equation*}
\begin{aligned}
\xi_{i,0} =(\psi_0,w_i)_{L^2}, \quad \xi_{i,1} =(\psi_1,w_i)_{L^2}, \quad \xi_{i,2} &=(\psi_2,w_i)_{L^2},
\end{aligned}
\end{equation*}
for $i \in [1,n]$.  In this way, we have by construction that
\begin{equation}  \label{GalerkinIC}
\begin{aligned}
\|\psi_0^n\|_{H^2} &\leq \|\psi_0\|_{H^2} &&\text{and} &&\psi_0^n \longrightarrow \psi_0  \text{ in } H_0^1\cap H^2, \\
\|\psi_1^n\|_{H^2} &\leq \|\psi_1\|_{H^2} &&\text{and} &&\psi_1^n \longrightarrow \psi_1 \text{ in } H_0^1 \cap H^2, \\
\|\psi_2^n\|_{H^1} &\leq \|\psi_2\|_{H^1} &&\text{and} &&\psi_2^n \longrightarrow \psi_1 \text{ in } H_0^1;
\end{aligned}
\end{equation}
see Lemma 7.5 in Ref.~\refcite{Robinson}. We then consider the following approximation of our problem
\begin{equation} \label{ibvp_semi-discrete}
\begin{aligned} 
\begin{cases}
(\tau \psi^n_{ttt}+\alpha \psi^n_{tt}-c^2 \Delta \psi^n-b \Delta \psi_t^n,  \phi)_{L^2} = (f, \phi)_{L^2}, \\[1mm]
\text{for every $\phi \in V_n$ pointwise a.e. in $(0,T)$}, \\[1mm]
(\psi^n(0), \psi_t^n(0), \psi^n_{tt}(0))=(\psi^n_0, \psi^n_1, \psi^n_2).
\end{cases}
\end{aligned}
\end{equation}
We introduce matrices $I^n=[I_{ij}]$, $M^n=[M_{ij}]$, $K^n=[K_{ij}]$, $C^n=[C_{ij}]$, and vector $F^n=[F_{i}]$, where
\begin{equation}
\begin{aligned}
& I^n_{ij}=(w_i, w_j)_{L^2}=  \delta_{ij}, \ M^n_{ij}(t)=(\alpha w_i, w_j)_{L^2}, \\
& K^n_{ij}=-c^2(\Delta w_i, w_j)_{L^2},\ D^n_{ij}=-b(\Delta w_i, w_j)_{L^2},\\
& F^n_i=(f, w_i)_{L^2}
\end{aligned}
\end{equation}
for $i,j \in [1,n]$, where $\delta_{ij}$ denotes the Kronecker delta. By setting $\xi^n=[\xi_{1} \ldots \xi_{n}]^T$, $\xi_0^n=[\xi_{1, 0} \ldots \xi_{n, 0}]^T$, $\xi_1^n=[\xi_{1, 1} \ldots \xi_{n,n}]^T$, and $\xi_2^n=[\xi_{1, 2} \ldots \xi_{n,2}]^T$,  problem \eqref{ibvp_semi-discrete} can be rewritten as
\begin{equation} \label{ODE_system}
\begin{aligned}
\begin{cases}
\tau I^n \xi^n_{ttt}+M^n \xi^n_{tt}+D^n \xi_t^n+K^n \xi^n=F^n(t), \\
(\xi^n(0), \xi^n_t(0), \xi^n_{tt}(0))=(\xi^n_0, \xi^n_1, \xi^n_2).
\end{cases}
\end{aligned}
\end{equation}
After additionally rewriting \eqref{ODE_system} as a first-order system, existence of an absolutely continuous solution $[\xi^n, \xi^n_t, \xi^n_{tt}]^{T}$ on $[0, T_n]$ for some $T_n \leq T$ follows from standard ODE theory; see, for example, Chapter 1 in Ref.~\refcite{Roubicek}. To see that $\xi^n \in H^3(0, T_n)$, we can employ a bootstrap argument,
\begin{equation}
\begin{aligned}
|\xi_{ttt}|^2_{L^2(0,T_n)}=& \tfrac{1}{\tau^2}|-M^n \xi^n_{tt}-D^n \xi_t^n-K^n \xi^n+F^n|^2_{L^2(0,T_n)} \\
\leq&\, C (\|\alpha\|^2_{L^\infty L^\infty}+\|f\|^2_{L^2L^2}).
\end{aligned}
\end{equation}
We, therefore, conclude that \eqref{ibvp_semi-discrete} has a solution $\psi^n \in H^3(0,T_n; V_n)$. The upcoming energy estimate will allow us to extend the existence interval to $[0,T]$.\\[3mm]
\noindent \textbf{Step 2: Energy estimates.} Our next goal is to obtain a bound for $\psi^n$ that is uniform with respect to $n$. To this end, we test our problem \eqref{ibvp_semi-discrete} with a suitable test function. \\[3mm]
\noindent \textbf{First estimate.} Testing the first equation in \eqref{ibvp_semi-discrete} with $\phi=-\Delta \psi^n_{tt}\in V_n$ and integrating over $(0,t)$, where $t \leq T_n$, yields the energy identity
\begin{equation}\label{enid1}
\begin{aligned}
& \tfrac{\tau}{2} |\nabla \psi^n_{tt}(t)|^2
+\tfrac{b}{2}|-\Delta \psi_t^n(t)|^2+\|\sqrt{\alpha}\nabla \psi^n_{tt}\|^2_{L_t^2L^2} \\
=&\, \tfrac{\tau}{2} |\nabla \psi^n_{tt}(0)|^2 
+\tfrac{b}{2}|-\Delta \psi_t^n(0)|^2-\int_0^t(\psi^n_{tt}\nabla\alpha,\nabla \psi^n_{tt})_{L^2} \textup{d}s\\
&-c^2 \left(-\Delta \psi^n,-\Delta \psi^n_{t}\right)_{L^2}\, \Bigr\vert_0^t+c^2\int_0^t\left(-\Delta \psi^n_{t},-\Delta \psi^n_{t}\right)_{L^2}\, \textup{d}s\\
&+\left(f, -\Delta \psi_t^n \right)_{L^2}\,\Bigr \vert_0^t-\int_0^t \left(f_t,-\Delta \psi_t^n \right)_{L^2}\, \textup{d}s 
=: \, \textbf{rhs}_1(t),
\end{aligned}
\end{equation}
where we have skipped the argument $(s)$ under the time integral for notational simplicity
and used the abbreviation $L_t^2L^2$ for $L^2(0,t;L^2(\Omega))$. 
To derive \eqref{enid1}, we have used the following three identities:
\[
\begin{aligned}
(\alpha \psi^n_{tt},-\Delta \psi^n_{tt})_{L^2}=(\alpha \nabla \psi^n_{tt},  \nabla \psi^n_{tt})_{L^2}+(\psi^n_{tt}\nabla\alpha,\nabla \psi^n_{tt}),
\end{aligned}
\]
and
\[
\begin{aligned}
&c^2\int_0^t\left(-\Delta \psi^n,-\Delta \psi^n_{tt}\right)_{L^2}\, \textup{d}s\\
=& \, c^2 \left(-\Delta \psi^n,-\Delta \psi^n_{t}\right)_{L^2}\, \Bigr\vert_0^t-c^2\int_0^t\left(-\Delta \psi^n_{t},-\Delta \psi^n_{t}\right)_{L^2}\, \textup{d}s,
\end{aligned}
\]
as well as 
\begin{equation*} 
\begin{aligned}
\int_0^t \left(f,-\Delta \psi^n_{tt} \right)_{L^2}\, \textup{d}s
=\left(f, -\Delta \psi_{t}^n \right)_{L^2}\,\Bigr \vert_0^t-\int_0^t \left(f_t,-\Delta \psi_t^n \right)_{L^2}\, \textup{d}s.
\end{aligned}
\end{equation*}
We note that $f \in H^1(0,T; L^2(\Omega)) \hookrightarrow C([0,T]; L^2(\Omega))$. We next estimate $\textup{rhs}_1(t)$ from above. We introduce here a constant that depends on the initial data to simplify the notation:
\[
\begin{aligned}
& C_1(\psi_0, \psi_1,\psi_2;\tau)\\\smallskip
=&\, \begin{multlined}[t]\tfrac{\tau}{2} |\nabla \psi^n_{tt}(0)|_{L^2}^2 
+\tfrac{b}{2}|-\Delta \psi_t^n(0)|_{L^2}^2+|f(0)|_{L^2}|-\Delta \psi_t^n(0)|_{L^2}\\[1mm]
+c^2 |-\Delta \psi^n(0)|_{L^2} |-\Delta \psi_t^n(0)|_{L^2} \end{multlined}\\
=&\, \tfrac{\tau}{2} |\nabla \psi^n_2(0)|_{L^2}^2 
+\tfrac{b}{2}|-\Delta \psi_1^n|_{L^2}^2+|f(0)|_{L^2}|-\Delta \psi_1^n|_{L^2}+c^2 |-\Delta \psi^n_0|_{L^2} |-\Delta \psi_1^n|_{L^2}.
\end{aligned}
\]
By applying H\"older's inequality, we get
\begin{equation*}
\begin{aligned}
\textbf{rhs}_1(t)
\leq& \, C_1(\psi_0, \psi_1,\psi_2;\tau)+\|\nabla \alpha\|_{L^\infty L^3}\|\psi_{tt}^n\|_{L^2 L^6}\|\nabla \psi_{tt}^n\|_{L_t^2L^2} \\
&+ c^2 |-\Delta \psi^n (t)|_{L^2} |-\Delta \psi_t^n(t)|_{L^2}+c^2 \|-\Delta \psi^n_t\|_{L_t^2L^2} \\
&+\|f\|_{L^\infty L^2}|-\Delta \psi_t^n(t)|_{L^2}+\|f_t\|_{L_t^2L^2}\|-\Delta \psi_t^n\|_{L_t^2L^2}.
\end{aligned}
\end{equation*}
We further estimate the right-hand side with the help of Young's $\varepsilon$-inequality \begin{align} \label{Young}
xy\leq \tfrac{\varepsilon}{2}x^2+\tfrac{1}{2\varepsilon}y^2,
\end{align}
and choosing $\varepsilon=\frac{b}{4}$ or $\varepsilon=1$, and the embedding results to obtain for a.e. $t \in [0, T_n]$,
\begin{equation}\label{estrhs1}
\begin{aligned}
\textbf{rhs}_1(t)
\leq& \, C_1(\psi_0, \psi_1,\psi_2;\tau) +C^{\Omega}_{H^1, L^6}\|\nabla \alpha \|_{L^\infty L^3}\|\nabla \psi_{tt}^n\|^2_{L_t^2L^2} \\
&+\tfrac{2c^4}{b} |-\Delta \psi^n(t)|^2_{L^2}+\tfrac{b}{8}|-\Delta \psi_t^n (t)|^2_{L^2}+c^2 \|-\Delta \psi^n_t\|_{L_t^2L^2} \\
&+\tfrac{2}{b} \|f\|^2_{L^\infty L^2}+\tfrac{b}{8}|-\Delta \psi_t^n(t) |^2_{L^2}
+\frac12 \|f_t\|^2_{L^2L^2}+\frac12 \|-\Delta \psi_t^n\|^2_{L_t^2L^2}.
\end{aligned}
\end{equation}
We can estimate the term $\|-\Delta \psi^n(t)\|_{L^2}$ appearing on the right-hand side as follows
\begin{align} \label{est_Delta_psi} 
\|-\Delta \psi^n\|_{L_t^\infty L^2} \leq \sqrt{t} \|-\Delta \psi_t^n\|_{L_t^2L^2}+|-\Delta \psi_0|_{L^2}.
\end{align}
Altogether, we get
\begin{equation}\label{est_1}
\begin{aligned}
& \tfrac{\tau}{2} |\nabla \psi^n_{tt}(t)|_{L^2}^2+\underline{\alpha}\|\nabla \psi^n_{tt}\|^2_{L_t^2L^2}
+\tfrac{b}{4}|-\Delta \psi^n_t(t)|_{L^2}^2\\\smallskip
\leq& \,\begin{multlined}[t] C_1(\psi_0, \psi_1,\psi_2;\tau)
+ C^\Omega_{H^1, L^6} \|\nabla \alpha\|_{L^\infty L^3}\|\nabla \psi_{tt}^n\|^2_{L_t^2L^2} \\
+ \tfrac{2c^4}{b} T \|-\Delta \psi_t^n\|^2_{L_t^2L^2}+ \tfrac{2c^4}{b}|-\Delta \psi_0|^2_{L^2}+c^2 \|-\Delta \psi^n_t\|_{L_t^2L^2} \\
+\tfrac{2}{b}\|f\|^2_{L^\infty L^2}+\tfrac12\|f_t\|^2_{L^2L^2}+\tfrac12\|-\Delta \psi_t^n\|^2_{L_t^2L^2}. \end{multlined}
\end{aligned}
\end{equation}
If the smallness assumption \eqref{gradalphasmall} holds, then the term containing $\|\nabla \alpha\|_{L^\infty L^3}$ can be absorbed into the left-hand side.\\[3mm]
\noindent \textbf{A priori bound for \mathversion{bold}$\psi^n$.} Applying Gronwall's inequality to \eqref{est_1}, and taking the supremum over $t\in(0,T_n)$ then yields
\begin{equation}
\begin{aligned}
&  \tau \|\nabla \psi^n_{tt}\|^2_{L^\infty L^2}+\|\nabla \psi^n_{tt}\|^2_{L^2L^2}+\|- \Delta \psi_t^n\|^2_{L^\infty L^2}\\[1mm]
\leq&\,\tilde{C}(\alpha, T_n, \tau)\left(|\psi^n_0|^2_{H^2}+|\psi^n_1|^2_{H^2}+\tau|\psi^n_2|^2_{H^1}+ \| f\|^2_{L^\infty L^2}+\|f_t\|^2_{L^2L^2}\right). 
\end{aligned}
\end{equation}
By employing the upper bounds for the approximate initial data stated in \eqref{GalerkinIC} and the inequality $T_n \leq T$, we further have
\begin{equation} \label{energy_est_discrete}
\begin{aligned}
&  \tau \|\nabla \psi^n_{tt}\|^2_{L^\infty L^2}+\|\nabla \psi^n_{tt}\|^2_{L^2L^2}+\|- \Delta \psi_t^n\|^2_{L^\infty L^2}\\[1mm]
\leq&\,\tilde{C}(\alpha, T_n, \tau)\left(|\psi_0|^2_{H^2}+|\psi_1|^2_{H^2}+\tau|\psi_2|^2_{H^1}+ \| f\|^2_{L^\infty L^2}+\|f_t\|^2_{L_t^2L^2}\right). 
\end{aligned}
\end{equation}
The constant above is given by
\begin{equation*}
\begin{aligned}
\tilde{C}(\alpha, T, \tau)= \tilde{C}_1 \, \textup{exp}\,\left(\tilde{C}_2 \, \left(\tfrac{1}{\tau}\|\nabla \alpha\|^2_{L^\infty L^3} +1+T \right) T \right),
\end{aligned}
\end{equation*}
or, if assumption \eqref{gradalphasmall} holds, by 
\begin{equation*}
\begin{aligned}
\tilde{C}(\alpha, T, \tau)= \tilde{C}(T)=\tilde{C}_1 \, \textup{exp}\,\left(\tilde{C}_2 \, \left(1+T \right) T \right),
\end{aligned}
\end{equation*}
where $\tilde{C}_1$, $\tilde{C}_2>0$ do not depend on $n$ or $\tau$. 
Since the right-hand side of \eqref{energy_est_discrete} does not depend on $T_n$, we can extend the existence interval to $[0,T]$, i.e. $T_n=T$.\\[3mm]
\noindent \textbf{Second estimate.} By testing \eqref{ibvp_semi-discrete} with $\phi=\tau \psi^n_{ttt} \in V_n$ and integrating over $(0,T)$, we obtain
\begin{equation}
\begin{aligned}
\tau^2\|\psi^n_{ttt}\|^2_{L^2L^2} 
\leq \|-\alpha \psi^n_{tt}+c^2 \Delta \psi^n+b \Delta \psi_t^n+f\|_{L^2L^2}\|\tau \psi^n_{ttt}\|_{L^2L^2},
\end{aligned}
\end{equation}
from which we have
\begin{equation} \label{bound_psi_ttt}
\begin{aligned}
& \tau\|\psi^n_{ttt}\|_{L^2L^2} \\
\leq& \, \|\alpha\|_{L^\infty L^\infty}\|\psi^n_{tt}\|_{L^2L^2}+c^2\|-\Delta \psi^n\|^2_{L^2 L^2}
+b\|-\Delta \psi_t^n\|_{L^2L^2}+\|f\|_{L^2L^2}.
\end{aligned}
\end{equation}
The terms $\|-\Delta \psi^n\|^2_{L^2L^2}$, $\|-\Delta \psi^n_t\|^2_{L^2L^2}$ can be further estimated similarly to \eqref{est_Delta_psi}, 
\begin{align*}
\|-\Delta \psi^n\|_{L^2 L^2} 
= & 
\left(\int_0^T \left| -\Delta \psi^n_0 +\int_0^t  -\Delta \psi^n_t(s)\, ds\right|_{L^2}^2\,\,dt\right)^{1/2}\\
\leq & 
\, \sqrt{T}|-\Delta \psi^n_0|_{L^2(\Omega)}+
\left(\int_0^T \left|\int_0^t  -\Delta \psi^n_t(s)\, ds\right|_{L^2}^2\,\,dt\right)^{1/2}
\\
\leq & 
\, \sqrt{T}|-\Delta \psi^n_0|_{L^2(\Omega)}+
\left(\int_0^T t^2 \, dt \right)^{1/2} \|-\Delta \psi^n_t\|_{L^\infty L^2}\\
\leq& \, \sqrt{T}|-\Delta \psi^n_0|_{L^2(\Omega)}+\sqrt{\tfrac{T^3}{3}}\|-\Delta \psi^n_t\|_{L^\infty L^2}, \\[1ex]
\|-\Delta \psi^n_t\|_{L^2 L^2} 
\leq& \, \sqrt{T}\|-\Delta \psi^n_t\|_{L^\infty L^2}.
\end{align*}
The term $\|\psi^n_{tt}\|_{L^2L^2}$ by means of the Poincar\'{e}-Friedrichs inequality, so that by using \eqref{energy_est_discrete} we obtain energy estimate \eqref{energy_est_lin} with $\psi^n$ in place of $\psi$.
\\[3mm] 
\noindent \textbf{Step 3: Passing to the limit.} On account of estimate \eqref{energy_est_discrete} and standard compactness results, together with the fact that the spatial and temporal domains $(0,T)$ and $\Omega$ are bounded, we know that there exist a subsequence, denoted again by $\{\psi^n\}_{n \in \mathbb{N}}$, and a function $\psi$ such that
\begin{equation} \label{weak_limits}
\begin{alignedat}{4} 
\psi_{ttt}^n  &\relbar\joinrel\rightharpoonup \psi_{ttt} &&\text{ weakly}  &&\text{ in } &&L^2(0,T; L^2(\Omega)),  \\
\psi_{tt}^n  &\relbar\joinrel\rightharpoonup \psi_{tt} &&\text{ weakly-$\star$}  &&\text{ in } &&L^\infty(0,T;H_0^1(\Omega)),  \\
\psi_t^n &\relbar\joinrel\rightharpoonup \psi_t &&\text{ weakly-$\star$} &&\text{ in } &&L^\infty(0,T; H_0^1(\Omega)\cap H^2(\Omega)),\\
\psi^n &\relbar\joinrel\rightharpoonup \psi &&\text{ weakly-$\star$} &&\text{ in } &&L^\infty(0,T; H_0^1(\Omega)\cap H^2(\Omega)).
\end{alignedat} 
\end{equation}
Our next task is to prove that $\psi$ solves \eqref{ibvp_linear}. We test \eqref{ibvp_semi-discrete} with $\eta \in C_c^\infty(0,T)$ and integrate over time to obtain 
\begin{equation} \label{1}
\begin{aligned}
&-\int_{0}^T(\tau \psi^n_{ttt}, w_i)_{L^2}\, \eta(t)\, \textup{d}t\\
=&\, -\int_{0}^T(\alpha \psi^n_{tt}-c^2\Delta \psi^n- b\Delta \psi_t^n -f, w_i)_{L^2} \eta(t) \, \textup{d}t,
\end{aligned}
\end{equation}
for all $i \in [1, n]$. Thanks to \eqref{weak_limits}, letting $n \rightarrow \infty$ in \eqref{1} leads to
\begin{equation*} 
\begin{aligned}
-\int_{0}^T(\tau \psi_{ttt}, w_i)_{L^2} \eta^{\prime}(t)\, \textup{d}t
=\, -\int_{0}^T(\gamma \psi_{tt}-c^2\Delta \psi-b\Delta \psi_{t} -f, w_i)_{L^2} \eta(t) \, \textup{d}t,
\end{aligned}
\end{equation*}
for all  $i \in \mathbb{N}$ and $\eta \in C^\infty(0,T)$. By construction, $\cup_{n \in \mathbb{N}}V_n$ is dense in $L^2(\Omega)$, so $\psi$ solves the PDE in \eqref{ibvp_linear} in the $L^2(0,T; L^2(\Omega))$ sense. Due to the embeddings
\begin{align*}
&	\psi^n \in W^{1,\infty}(0,T; H_0^1(\Omega) \cap H^2(\Omega)) \hookrightarrow \hookrightarrow	C([0,T]; H_0^1(\Omega) \cap H^2(\Omega)),\\
& \psi_t^n \in W^{1, \infty}(0,T; H_0^1(\Omega)) \hookrightarrow \hookrightarrow C([0,T]; H_0^1(\Omega)), \\
& \psi_{tt}^n \in H^1(0,T; L^2(\Omega)) \hookrightarrow C([0,T]; L^2(\Omega)), 
\end{align*}
we know that
\begin{align*}
&\psi^n(0) \rightarrow \psi(0) \quad \text{in } H_0^1(\Omega) \cap H^2(\Omega),\\
& \psi_t^n(0) \rightarrow \psi_t(0) \quad \text{in } H_0^1(\Omega),\\
& \psi_{tt}^n(0) \rightarrow \psi_{tt}(0) \quad \text{in } L^2(\Omega).
\end{align*}
Thanks to \eqref{GalerkinIC}, we can then infer that $\psi(0)=\psi_0$, $\psi_t(0)=\psi_1$, and $\psi_{tt}(0)=\psi_2$.  Altogether, we conclude that $\psi$ is a solution of the initial-boundary value problem \eqref{ibvp_linear}. \\[3mm]
\noindent \textbf{Step 4: Energy inequality for \mathversion{bold}$\psi$.} We can take the limit inferior as $n \rightarrow \infty$ of \eqref{energy_est_discrete}, \eqref{bound_psi_ttt}, and via the weak and the weak-$\star$ lower semi-continuity of norms obtain the final estimate \eqref{energy_est_lin}. Uniqueness of a solution follows by the linearity of the equation, together with the fact that the homogeneous equation only has the zero solution, by the above energy estimates.
\end{proof}

\section{Well-posedness of the nonlinear Westervelt-type wave equation for $\tau>0$}\label{sec:wellposed_Wes}

After having studied the linearized equation, we now proceed to the nonlinear model \eqref{WesterveltMC}. For proving well-posedness of \eqref{WesterveltMC}, we introduce the fixed-point operator $\mathcal{T}$ that maps $\phi$ to a solution $\psi$ of 
\begin{equation}\label{PDEmathcalT}
\tau\psi_{ttt}+(1-k\phi_t)\psi_{tt}-c^2\Delta \psi - b\Delta \psi_t
= 0\,,
\end{equation}
on some ball 
\begin{equation}\label{Brho_W}
\begin{aligned}
B_\rho^{X^W}=\{\psi\in X^W \, :& \ \psi(0)=\psi_0\,, \ \psi_t(0)=\psi_1\,, \ \psi_{tt}(0)=\psi_2\,, \\ 
&\|\psi\|_{X^W}^2:=\tau^2 \|\psi_{ttt}\|^2_{L^2L^2} + \tau \|\psi_{tt}\|^2_{L^\infty H^1}\\
&\hspace{1.8cm} +\|\psi_{tt}\|^2_{L^2 H^1}+\|\psi\|^2_{W^{1,\infty} H^2}
\leq \rho^2\}
\end{aligned}
\end{equation}
in the space $X^W$, defined in \eqref{XWXKspaces}. Note that the operator is well-defined on account of Theorem~\ref{th:wellposedness_lin}. \\
\indent For establishing $\mathcal{T}$ as a self-mapping on $B_\rho^{X^W}$, it is crucial to prove that $\alpha=1-k\phi_t$ is in 
$L^\infty(0,T; L^\infty(\Omega))\cap L^\infty(0,T; W^{1,3}(\Omega))$
and that the smallness condition \eqref{gradalphasmall} of Theorem~\ref{th:wellposedness_lin} holds, provided $\phi\in B_\rho^X$. Smallness of $\phi$ will also be required for verifying the non-degeneracy condition $\alpha(t)\geq\ul{\alpha}>0$.\\
\indent Note that the radius of the neighborhood in which the self-mapping property holds will be independent of $\tau$. In particular, it holds for arbitrarily small $\tau$ and therefore allows for taking limits as $\tau\to0$ later on. \\
\indent Contractivity of $\mathcal{T}$, based on the fact that 
$\hat{\psi}=\psi_1-\psi_2=\mathcal{T}(\phi_1)-\mathcal{T}(\phi_2)$ solves
\begin{equation}\label{hat_Wes}
\begin{aligned}
&\tau\hat{\psi}_{ttt}+(1-k\phi_{1\,t})\hat{\psi}_{tt}-c^2\Delta \hat{\psi} - b\Delta \hat{\psi}_t = k\hat{\phi}_t \psi_{2\,tt}\,,
\end{aligned}
\end{equation}
with homogeneous initial and boundary conditions (where $\hat{\phi}=\phi_1-\phi_1$),
would require to prove that $\alpha_1=1-k\phi_{1\,t}$ and $f_2=k\hat{\phi}_t \psi_{2\,tt}$ are in 
$L^\infty(0,T; L^\infty(\Omega))\cap L^\infty(0,T; W^{1,3}(\Omega))$
and $H^1(0,T;L^2(\Omega))$, respectively. This regularity, however, would only be possible in an $O(\sqrt{\tau})$ neighborhood because of the $\phi_{2ttt}$ term arising in $f_{2t}$.
Therefore, we do not prove contractivity, but base our existence proof on Schauder's theorem, similarly to the approach in Ref.~\refcite{KT18_ModelsNlAcoustics}.
\begin{theorem} \label{th:wellposedness_Wes}
	Let $c^2$, $b>0$, $k\in\mathbb{R}$ and let $T>0$. Then there exist $\rho>0$ and $\rho_0>0$ such that for all $(\psi_0,\psi_1,\psi_2)\in X^W_0
=H_0^1(\Omega) \cap H^2(\Omega)\times H_0^1(\Omega) \cap H^2(\Omega)\times H_0^1(\Omega)
$ satisfying 
\begin{equation}\label{smallness_init_Wes}
\|\psi_0\|_{H^2(\Omega)}^2+\|\psi_1\|_{H^2(\Omega)}^2+\tau\|\psi_2\|_{H^1(\Omega)}^2\leq\rho_0^2\,,
\end{equation} 
there exists a solution $\psi\in X^W$ of 
	\begin{equation} \label{ibvp_Westervelt_MC}
	\begin{aligned}
	\begin{cases}
	\tau\psi_{ttt}+\psi_{tt}-c^2\Delta \psi - b\Delta \psi_t = \left(
\tfrac{k}{2}
(\psi_t)^2\right)_t\,
 \quad \mbox{ in }\Omega\times(0,T), \\[1mm]
	\psi=0 \quad \mbox{ on } \partial \Omega\times(0,T),\\[1mm]
	(\psi, \psi_t, \psi_{tt})=(\psi_0, \psi_1, \psi_2) \quad \mbox{ in }\Omega\times \{0\},
	\end{cases}
	\end{aligned}
	\end{equation}
\end{theorem}
\noindent such that it holds
\begin{align} \label{est_Westervelt}
\tau^2 \|\psi_{ttt}\|^2_{L^2L^2} + \tau \|\psi_{tt}\|^2_{L^\infty H^1}+\|\psi_{tt}\|^2_{L^2 H^1}+\|\psi\|^2_{W^{1,\infty} H^2}
\leq \rho^2\,.
\end{align}
\begin{proof}
Our proof relies on Schauder's fixed-point theorem applied to the operator $$\mathcal{T}: B_\rho^{X^W}\ni \phi \mapsto \psi,$$
where $\psi$ solves \eqref{PDEmathcalT}. To obtain the self-mapping property of $\mathcal{T}$ , we have to verify the condition \eqref{gradalphasmall} of Theorem~\ref{th:wellposedness_lin} as well as $\alpha(t)\geq\ul{\alpha}>0$. We thus estimate the 
$L^\infty(0,T; L^\infty(\Omega))\cap L^\infty(0,T; W^{1,3}(\Omega))$
norm of $\alpha=1-k\phi_t$. In view of the bounds
\[
\begin{aligned} 
\|\nabla\alpha\|_{L^\infty L^3}
&=|k|\,\|\nabla \phi_t\|_{L^\infty L^3}
\leq |k|\, C_{H^2, W^{1,3}}^\Omega \|\phi\|_{W^{1,\infty} H^2}
\leq |k|\, C_{H^2, W^{1,3}}^\Omega \rho,\\
\|\alpha-1\|_{L^\infty L^\infty}
&=|k|\,\|\phi_t\|_{L^\infty L^\infty}
\leq |k|\, C_{H^2, L^\infty}^\Omega \|\phi\|_{W^{1,\infty} H^2}
\leq |k|\, C_{H^2, L^\infty}^\Omega \rho,
\end{aligned}
\]
the smallness condition \eqref{gradalphasmall} and the non-degeneracy condition $\alpha(t)\geq\ul{\alpha}>0$ can be satisfied by choosing 
\[
\rho<\left(2|k|\, \max \, \left\{C_{H^2, L^\infty}^\Omega\, , \ C_{H^1, L^6}^\Omega C_{H^2, W^{1,3}}^\Omega \right \}\right)^{-1}. 
\] 
The self-mapping property follows from the estimate \eqref{energy_est_lin}, with $f=0$ and
\[
\rho_0^2\leq \left(C_1 \left(\tfrac{13}{4}+T^3\right) \textup{exp}\,\left(C_2 \, \left(1+T\right) T\right)\right)^{-1}\, \rho^2.
\] 
The set $B_\rho^{X^W}$ is a weak* compact and convex subset of the Banach space $X^W$, defined in \eqref{XWXKspaces}. The weak* continuity of $\mathcal{T}$ can be established as follows: For any sequence $(\phi_n)_{n\in\mathbb{N}}\subseteq B_\rho^{X^W}$ that weakly* converges to $\phi\in B_\rho^{X^W}$ in $X^W$, we also have $$(\mathcal{T}(\phi_n))_{n\in\mathbb{N}}\subseteq B_\rho^{X^W}.$$ Thus, by compactness of the embedding $X^W\to W^{1,\infty}(0,T;L^\infty(\Omega))$, there exists a subsequence $(\phi_{n_\ell})_{\ell\in\mathbb{N}}$ such that $1-k\phi_{n_\ell\, t}$ converges to $1-k\phi_t$ strongly in $L^\infty(0,T;L^\infty(\Omega))$ and $\mathcal{T}(\phi_{n_\ell})$ converges weakly* in $X^W$ to some $\psi\in B_\rho^{X^W}$, which by definition of $B_\rho^{X^W}$ satisfies the initial and homogeneous Dirichlet boundary conditions. It is readily checked that $\psi$ also solves the PDE \eqref{PDEmathcalT}, which, by uniqueness in Theorem \ref{th:wellposedness_lin}, implies $\psi=\mathcal{T}(\phi)$. A subsequence-subsequence argument yields weak* convergence in $X^W$ of the whole sequence $(\mathcal{T}(\phi_n))_{n\in\mathbb{N}}$ to $\mathcal{T}(\phi)$. \\
\indent We can therefore conclude existence of a fixed point of $\mathcal{T}$ in $B_\rho^{X^W}$ from the general version of Schauder's fixed-point theorem in locally convex topological spaces; see Ref~\refcite{Fan1952}, which we here quote for the convenience of the reader
\begin{quote}
Let $L$ be a locally convex topological linear space and $K$ a compact convex set in $L$. Let $M(K)$ be the family of all closed convex (non-
empty) subsets of $K$. Then for any upper semicontinuous point-to-set transformation $f$ from $K$ into $M (K)$, there exists a point $x_0\in K$ such that $x_0\in f(x_0)$.
\end{quote}
We use this theorem with the single valued point-to-set relation (i.e., mapping) $f=\mathcal{T}$, the weak*topology on $X^W$, and $K=B_\rho^{X^W}$. 
\end{proof}



\section{Higher energy estimates} \label{sec:enest_higher}
Due to the appearance of $\|f_t\|_{L^2L^2}$ on the right-hand side of estimate \eqref{energy_est_lin} in Theorem~\ref{th:wellposedness_lin}, we cannot rely only on this estimate
for the Kuznetsov-type JMGT equation \eqref{KuznetsovMC} since $$f_t=2\nabla\phi\cdot\nabla\phi_{tt}+2|\nabla\phi_t|^2.$$ 
Existence of solutions can still be based on Theorem \ref{th:wellposedness_lin}, (case $\|\nabla\alpha\|_{L^\infty L^3} <\ul{\alpha}/C^\Omega_{H^1, L^6}$ with a $\tau$-independent bound on the energy) because $f$ is still in the right space. However, $$\|f_t\|_{L^2 L^2}=2\|\nabla \psi\cdot\nabla\psi_{tt} + |\nabla \psi_t|^2\|_{L^2 L^2}$$ can only be shown to be bounded by $\frac{1}{\sqrt{\tau}}$, so it might be large as $\tau\to0$. This does not matter for proving existence according to Theorem \ref{th:wellposedness_lin}, but excludes a fixed-point argument for proving well-posedness of the nonlinear equation \eqref{KuznetsovMC} in this setting.\\
\indent To be able to take limits as $\tau\to0$, we thus need higher order energy estimates.  In particular, we need to derive $\tau$-independent bounds on $\|\psi\|_{L^\infty H^3}$ which will enable us to estimate $f=2\nabla \psi\cdot\nabla\psi_t$ in the required norms. We replace estimate \eqref{energy_est_lin}  by an estimate on the auxiliary function $z$ and complement this with a higher order in space estimate on $\psi$.
   
In order to carry out these new error estimates, we now turn our attention to studying the equation \eqref{WesterveltMC_lin}, restated again here for convenience:
\begin{equation} \tag{\ref{WesterveltMC_lin}}
\begin{aligned}
\tau\psi_{ttt}+\alpha(x,t)\psi_{tt}-c^2\Delta \psi - b\Delta \psi_t = f \mbox{ in }\Omega\times(0,T)\,, 
\end{aligned}
\end{equation}
together with its equivalent reformulation \eqref{dampedwave_z} using \eqref{eq:z}, \eqref{eq:gamma}, i.e., 
\begin{equation} \tag{\ref{dampedwave_z}}
\begin{aligned}
\tau z_{tt} +\gamma z_t - b \Delta z-\gamma\tfrac{c^2}{b} z +\gamma\tfrac{c^4}{b^2} \psi=f \mbox{ in }\Omega\times(0,T)\,,
\end{aligned}
\end{equation}
where we recall that the auxiliary state is given by
\begin{equation} \tag{\ref{eq:z}}
z=\psi_t+\tfrac{c^2}{b}\psi.
\end{equation}
We assume that for some $\ul{\alpha}>0$, $\ul{\gamma}>0$
\begin{equation}\label{eq:alphagamma_pos}
\alpha(t)\geq\ul{\alpha},\,\quad \gamma(t)=\alpha(t)-\tau\tfrac{c^2}{b}\geq\ul{\gamma} \quad \mbox{ on }\Omega \mbox{ for a.e. }t\in(0,T).
\end{equation}
\begin{theorem} \label{Thm:LinearHigher}
	Let $c^2$, $b$, and let $T>0$. Assume that
	\begin{itemize}
		\item $\alpha \in 
W^{1,1}(0,T;H^1(\Omega))\cap L^\infty(0,T;W^{1,3}(\Omega)\cap L^1(0,T;H^2(\Omega)$, \medskip
		\item $f\in H^1(0,T; L^2(\Omega))\cap L^2(0,T;H^1(\Omega))$, \medskip
		\item $(\psi_0, \psi_1, \psi_2)\in X^K_0 :=H_0^1(\Omega)\cap H^3(\Omega)\times H_0^1(\Omega)\cap H^2(\Omega)\times H_0^1(\Omega)$,
	\end{itemize}
and that the non-degeneracy condition \eqref{eq:alphagamma_pos} holds. Then there exists $\bar{\tau}>0$ such that for $\tau \in (0, \bar{\tau})$ and sufficiently small $\|\nabla \gamma\|_{L^\infty L^3}$
, there exists a unique solution $(\psi, z)$ of the problem 
	\begin{equation} \label{ibvp_linear_Kuz}
	\begin{aligned}
	\begin{cases}
	\tau z_{tt} +(\alpha-\tau \tfrac{c^2}{b}) z_t - b \Delta z-\gamma\tfrac{c^2}{b} z +\gamma\tfrac{c^4}{b^2} \psi=f \quad \mbox{ in }\Omega\times(0,T), \\[1mm]
	z=\psi_t+\tfrac{c^2}{b}\psi \quad \mbox{ in }\Omega\times(0,T), \\[1mm]
	\psi=0 \quad \mbox{ on } \partial \Omega\times(0,T),\\[1mm]
	(\psi, \psi_t, \psi_{tt})=(\psi_0, \psi_1, \psi_2) \quad \mbox{ in }\Omega\times \{0\},
	\end{cases}
	\end{aligned}
	\end{equation}
	that satisfies $(\psi,z)\in L^\infty(0,T;H^3(\Omega))\times (L^\infty(0,T;H^2(\Omega))\cap W^{1,\infty}(0;T;H_0^1(\Omega)))$; in other words,
	\begin{equation}\label{regularity_Kuz}
	\begin{aligned}
	\psi\in X^K:=L^\infty(0,T;H^3(\Omega))\cap W^{1,\infty}(0,T;H^2(\Omega))\cap W^{2,\infty}(0;T;H_0^1(\Omega)).
	\end{aligned}
	\end{equation}
Furthermore, there exists a $C(\gamma,T)>0$, which does not depend on $\tau$, such that
	\begin{equation}\label{reg_est_Kuz}
	\begin{aligned}
	&\|\psi\|_{L^\infty H^3}^2+\|\psi_t+\tfrac{c^2}{b}\psi \|_{L^\infty H^2 }^2
    +\|\psi_{tt}+\tfrac{c^2}{b}\psi_t \|_{L^2 H_0^1 }^2
	+\tau \|\psi_{tt}+\tfrac{c^2}{b}\psi_t \|_{L^\infty H_0^1 }^2\\
	\leq& \, C(\gamma,T) \left(|\psi_0|_{H^3}^2+|\psi_1|_{H^2}^2+\tau|\psi_2|_{H^1}^2+\|f\|_{L^2 H^1}^2\right).
	\end{aligned}
	\end{equation}
\end{theorem}
Note that while $f\in H^1(0,T; L^2(\Omega))\cap L^2(0,T;H^1(\Omega))$ is required, in the right-hand side of the energy estimate only $\|f\|_{L^2 H^1}$, but not $\|f\|_{H^1 L^2}$ appears.
\begin{proof}
\noindent \textbf{Step 1: Existence of a solution.}
Theorem \ref{th:wellposedness_lin} implies existence of a solution $(\psi,z)$ of \eqref{ibvp_linear_Kuz} with regularity as stated in \eqref{regularity}
and 
\[
	\begin{aligned}
	z \in \, & L^{\infty}(0;T;H_0^1(\Omega) \cap H^2(\Omega)) \cap W^{1, \infty}(0,T; H_0^1(\Omega))
	\cap H^2(0,T; L^2(\Omega)).
	\end{aligned}
\]
Note that $z_t$ inherits the homogeneous Dirichlet boundary conditions from $\psi$.
Thus it only remains to establish the higher order energy estimates.
For this purpose, we return to the Galerkin approximation \eqref{ibvp_semi-discrete} and define $z^n=\psi^n+\tfrac{c^2}{b}\psi^n$. \\[3mm]
\noindent \textbf{Step 2: A priori estimates.} As in Section \ref{sec:enest}, our goal is to obtain a bound for $\psi^n$ that is uniform with respect to $n$. To this end, we test the spatially discretized version of our problem \eqref{ibvp_linear_Kuz} with two test functions. \\[2mm]
\noindent \textbf{The first energy identity.} Problem \eqref{ibvp_semi-discrete}  can be equivalently rewritten as
\begin{equation} \label{ibvp_semi-discrete_equiv}
\begin{aligned} 
\begin{cases}
\left(\tau z^n_{tt}+\gamma z^n_{t}-b\Delta z^n - \tfrac{c^2}{b}\gamma z^n+\gamma 
\tfrac{c^4}{b^2} \psi^n, \phi \right)_{L^2} = (f, \phi)_{L^2}, \\[1mm]
\text{for every $\phi \in V_n$ pointwise a.e. in $(0,T)$}, \\[1mm]
z^n=\psi^n_t+\tfrac{c^2}{b} \psi^n, \\[1mm]
(\psi^n(0), \psi_t^n(0), \psi^n_{tt}(0))=(\psi^n_0, \psi^n_1, \psi^n_2),
\end{cases}
\end{aligned}
\end{equation}
where $V_n$ is defined as in the proof of Theorem \ref{th:wellposedness_lin}, as the span of the first $n$ eigenfunctions of the Laplacian with homogeneous Dirichlet boundary conditions. Multiplying the first equation in \eqref{ibvp_semi-discrete_equiv} by $-\Delta z^n_t\in V_n$ and integrating over $\Omega$ and $(0,t)$ yields the energy identity
\begin{equation}\label{enid1_Kuz}
\begin{aligned}
& \tfrac{\tau}{2} |\nabla z^n_t(t)|^2
+\int_0^t|\sqrt{\gamma}\nabla z^n_t|^2\, \textup{d}s 
+\tfrac{b}{2}|-\Delta z^n(t)|^2\\
=&\, \tfrac{\tau}{2} |\nabla z^n_t(0)|^2
-\int_0^t(z_t^n \, \nabla\gamma,\nabla z^n_t)_{L^2}\, \textup{d}s
+\tfrac{b}{2}|-\Delta z^n(0)|^2\\
&-\tfrac{c^2}{b}\int_0^t(\gamma \nabla z^n,\nabla z^n_t)_{L^2}\,\textup{d}s -\tfrac{c^2}{b}\int_0^t(z^n \nabla \gamma,\nabla z^n_t)_{L^2}\ \textup{d}s\\
&+\int_0^t(\nabla f-\gamma\tfrac{c^4}{b^2}\nabla \psi^n-\tfrac{c^4}{b^2}\psi^n \nabla\gamma,\nabla z^n_t)_{L^2}\, \textup{d}s\\
&-\int_0^t \langle \nu\cdot \nabla z^n_t, f\rangle_{H^{-1/2}(\partial\Omega),H^{1/2}(\partial\Omega)} \, \textup{d}s
=: \, \textbf{rhs}_1(t),
\end{aligned}
\end{equation}
where we have skipped the argument $(s)$ under the time integral for notational simplicity. To derive \eqref{enid1_Kuz}, we have used the identity
\[
\begin{aligned}
(\gamma z^n_t,-\Delta z^n_t)_{L^2}=|\sqrt{\gamma} \nabla z^n_t|^2+(z^n_t\nabla\gamma,\nabla z^n_t)_{L^2}.
\end{aligned}
\]
Furthermore, we have made use of 
\[
\begin{aligned}
-\int_0^t(\gamma\tfrac{c^2}{b} z^n,-\Delta z^n_t)_{L^2}\, \textup{d}s
=\,
-\int_0^t  (\gamma\tfrac{c^2}{b} \nabla z^n,\nabla z^n_t )_{L^2}\, \textup{d}s-\int_0^t (z^n \nabla \gamma \tfrac{c^2}{b},\nabla z^n_t)_{L^2}\, \textup{d}s
\end{aligned}
\]
as well as the identity
\begin{equation*}
\begin{aligned}
\int_0^t (f-\gamma\tfrac{c^4}{b^2}\psi^n,-\Delta z^n_t)_{L^2}\, \textup{d}s
=& \,\begin{multlined}[t] \int_0^t(\nabla f-\gamma\tfrac{c^4}{b^2}\nabla \psi^n-\tfrac{c^4}{b^2}\psi^n \nabla\gamma,\nabla z^n_t)_{L^2}\, \textup{d}s\\
-\int_0^t \langle \nu\cdot \nabla z^n_t, f\rangle_{H^{-1/2}(\partial\Omega),H^{1/2}(\partial\Omega)} \, \textup{d}s.\end{multlined}
\end{aligned}
\end{equation*}

\noindent \textbf{The second energy identity.}	Our aim is to obtain a bound on $\psi$ in the $H^3(\Omega)$ norm. To this, end we  test \eqref{ibvp_semi-discrete} with $\phi=(-\Delta)^2 \psi^n\in V_n$ (due to the fact that $\psi^n$ is a linear combination of eigenfunctions of $-\Delta$) which yields the second energy identity
\begin{equation}\label{enid2_Kuz}
\begin{aligned}
& c^2\int_0^t|\nabla(-\Delta)\psi^n|^2\, \textup{d}s
+\tfrac{b}{2}|\nabla(-\Delta)\psi^n(s)|^2 \, \Bigl \vert_0^t\\
=& \, -\tau (\nabla\psi^n_{tt}(s),\nabla(-\Delta)\psi^n(s))_{L^2} \, \Bigl \vert_0^t 
+\tfrac{\tau}{2}|-\Delta\psi^n_t(s)|^2 \, \Bigl \vert_0^t \\
&+\int_0^t (-\Delta[\alpha\psi^n_t],-\Delta\psi^n_t)_{L^2}\,\textup{d}s  \\
&-(\alpha(s)\nabla\psi^n_t(s)+\psi_t(s)\nabla\alpha(s),\nabla(-\Delta)\psi^n(s))_{L^2}\,  \Bigl \vert_0^t\\
&+\int_0^t (\psi^n_t\nabla\alpha_t+\alpha_t\nabla\psi^n_t,\nabla(-\Delta)\psi^n)_{L^2}\, \textup{d}s  
+\int_0^t (\nabla f,\nabla(-\Delta)\psi^n)_{L^2}\, \textup{d}s\\
&-\int_0^t \langle \nu\cdot \nabla (-\Delta)\psi^n, f\rangle_{H^{-1/2}(\partial\Omega),H^{1/2}(\partial\Omega)} \, \textup{d}s\,=: \, \textbf{rhs}_2(t),
\end{aligned}
\end{equation}
Above, we have made use of
\[
\begin{aligned}
(\psi^n_{ttt},(-\Delta)^2\psi^n)_{L^2}=&\, (\nabla\psi^n_{ttt},\nabla(-\Delta)\psi^n)_{L^2}\\
=& \,\frac{\textup{d}}{\textup{d}t}\Bigl[(-\nabla\psi^n_{tt},\nabla(-\Delta)\psi^n)_{L^2}-\tfrac12|-\Delta\psi^n_t|^2\Bigr],\\
\end{aligned}
\]
and the fact that $-\Delta \psi^n_t=0$ on $\partial \Omega$. Morever, we rewrote the $\alpha$ term as follows
\begin{equation*}
\begin{aligned}
&\int_0^t (\alpha\psi^n_{tt},(-\Delta)^2\psi^n)_{L^2}\, \textup{d}s \\
\stackrel{\text{i.b.p. in space}}{=}& \, \int_0^t (\nabla[[\alpha\psi^n_t]_t-\alpha_t\psi_t],\nabla(-\Delta)\psi^n)_{L^2}\, \textup{d}s\\ 
\stackrel{\text{i.b.p. in time}}{=}&\, (\nabla[\alpha(s)\psi^n_t(s)],\nabla(-\Delta)\psi^n(s))_{L^2} \, \Bigl \vert_0^t\\
&-\int_0^t (\nabla[\alpha\psi^n_t],\nabla(-\Delta)\psi^n_t)_{L^2}\,  \textup{d}s
-\int_0^t (\nabla[\alpha_t\psi^n_t],\nabla(-\Delta)\psi^n)_{L^2}\, \textup{d}s \\
\stackrel{\text{i.b.p. in space}}{=}& \,(\nabla[\alpha(s)\psi^n_t(s)],\nabla(-\Delta)\psi^n(s))_{L^2} \, \Bigl \vert_0^t\\
& -\int_0^t (-\Delta[\alpha\psi^n_t],-\Delta\psi^n_t)_{L^2}\,\textup{d}s 
-\int_0^t (\nabla[\alpha_t\psi^n_t],\nabla(-\Delta)\psi^n)_{L^2}\, \textup{d}s,
\end{aligned}
\end{equation*}
where we used again that $-\Delta \psi_t^n=0$ on $\partial \Omega$. Note that under the assumptions made on $\alpha$, we have that $(\alpha\psi^n_t)(t)\in H_0^1(\Omega) \cap H^2(\Omega)$ for almost every $t\in(0,T)$, since $\psi^n_t(t)\in H_0^1(\Omega) \cap H^2(\Omega)$. \\
\indent The left-hand sides of our two energy identities \eqref{enid1_Kuz} and \eqref{enid2_Kuz} can be estimated from below by 
\begin{equation} \label{lhs_1}
\begin{aligned}
&
\tfrac{\tau}{2} |\nabla z^n_t(t)|^2
+\int_0^t|\sqrt{\gamma}\nabla z^n_t|^2\, \textup{d}s 
+\tfrac{b}{2}|-\Delta z^n(t)|^2 \\
\geq&\, \tfrac18\left (\tau \|\nabla z^n_t\|^2_{L_t^\infty L^2}+\ul{\gamma} \|\nabla z^n_t\|^2_{L_t^2 L^2}+ b \|-\Delta z^n\|^2_{L_t^\infty L^2}\right),
\end{aligned}
\end{equation}
and by
\begin{equation} \label{lhs_2}
\begin{aligned}
&
c^2\int_0^t|\nabla(-\Delta) \psi^n|^2\, \textup{d}s
+\tfrac{b}{2}|\nabla(-\Delta) \psi^n(t)|^2 \\
\geq& \, \tfrac14 \, \left(2c^2\|\nabla(-\Delta) \psi^n\|^2_{L_t^2 L^2}
+ b \|\nabla(-\Delta)\psi^n\|^2_{L_t^\infty L^2}\right),
\end{aligned}
\end{equation}
respectively. 

We will consider the weighted sum \eqref{enid1_Kuz} plus $\lambda>0$ times \eqref{enid2_Kuz}, which therefore can be bounded from below by  
\begin{equation}\label{lhs}
\begin{aligned}
\textbf{lhs}(t)=& \,\begin{multlined}[t] \tfrac18 \left(\tau \|\nabla z^n_t\|^2_{L_t^\infty L^2}+\ul{\gamma} \|\nabla z^n_t\|^2_{L_t^2 L^2}+ b \|-\Delta z^n\|^2_{L_t^\infty L^2}\right) \\
+ \tfrac{\lambda}{4}\left (2c^2\|\nabla(-\Delta) \psi^n\|^2_{L_t^2 L^2}
+  b \|\nabla(-\Delta)\psi^n\|^2_{L_t^\infty L^2}\right). \end{multlined}
\end{aligned}
\end{equation}
It then remains to estimate the right-hand sides, $\text{rhs}_1(t)$ and $\lambda \text{rhs}_2(t)$. \\[2mm]
\noindent \textbf{Estimates of the right-hand sides.} For estimating the right-hand sides in \eqref{enid1_Kuz} and \eqref{enid2_Kuz}, we can then use the norms of $z$, $z_t$, and $\psi$ appearing in the lower bounds \eqref{lhs_1} and \eqref{lhs_2}. Furthermore, we can employ the continuous embeddings \eqref{embeddigs}
as well as boundedness of $(-\Delta)^{-1}:L^2(\Omega)\to H_0^1(\Omega)\cap H^2(\Omega)$ and the Poincar\'{e}-Friedrichs inequality \eqref{ellreg-PF}.
Additionally, we employ the identities $$\psi_t=z-\tfrac{c^2}{b}\psi, \quad \psi_{tt}=z_t-\tfrac{c^2}{b}z+\tfrac{c^4}{b^2}\psi.$$ \\
\indent To simplify the notation, we introduce two constants depending on the initial data,
\[
\begin{aligned}
C_1(\psi_0,\psi_1,\psi_2;\tau)
=&\, \tfrac{\tau}{2} |\nabla z_t(0)|^2
+\tfrac{b}{2}|-\Delta z(0)|^2\\
=&\, \tfrac{\tau}{2} |\nabla \psi_2+\tfrac{c^2}{b}\nabla\psi_1|^2
+\tfrac{b}{2}|-\Delta \psi_1-\tfrac{c^2}{b}\Delta\psi_0 |^2, 
\end{aligned}
\]
as well as
\[
\begin{aligned}
C_2(\psi_0,\psi_1,\psi_2;\tau)=& \,
\tfrac{b}{2}\left|\nabla(-\Delta)\psi_0 \right|^2 
+\tau (\nabla\psi_2,\nabla(-\Delta)\psi_0)_{L^2}
-\tfrac{\tau}{2} \left|-\Delta\psi_1 \right|^2\\
&+(\alpha(0)\nabla\psi_1+\psi_1\nabla\alpha(0),\nabla(-\Delta)\psi_0)_{L^2}.
\end{aligned}
\]
\indent By applying H\"older's inequality and the trace theorem, we get for the right-hand side in \eqref{enid1_Kuz},
\begin{equation*}
\begin{aligned}
&\textbf{rhs}_1(t)\\
\leq& \, C_1(\psi_0,\psi_1,\psi_2;\tau)+\|\nabla z^n_t\|_{L_t^2 L^2} \|z^n_t\|_{L_t^2 L^6} \|\nabla\gamma\|_{L^\infty L^3}\\
&+\|\nabla z^n_t\|_{L_t^2 L^2} \|\nabla z^n\|_{L_t^2 L^6} \tfrac{c^2}{b}\|\gamma\|_{L^\infty L^3}
+\|\nabla z^n_t\|_{L_t^2 L^2} \|z^n\|_{L_t^2 L^\infty} \tfrac{c^2}{b}\|\nabla\gamma\|_{L^\infty L^2}\\
&+\|\nabla z^n_t\|_{L_t^2 L^2} \Bigl(
\|\nabla f\|_{L^2 L^2}
+\tfrac{c^4}{b^2} \|\gamma\|_{L^\infty L^2} \|\nabla \psi^n\|_{L_t^2 L^\infty}
+\tfrac{c^4}{b^2} \|\nabla \gamma\|_{L^\infty L^2} \|\psi^n\|_{L_t^2 L^\infty}\Bigr)\\
&+\|\nu\cdot\nabla z^n_t\|_{L_t^2 H^{-1/2}(\partial\Omega)} \|f\|_{L^2 H^{1/2}(\partial\Omega)},
\end{aligned}
\end{equation*}
a.e. in time. We further obtain, by making use of Young's inequality \eqref{Young} and the embedding results,
\begin{equation}\label{estrhs1_Kuz}
\begin{aligned}
&\textbf{rhs}_1(t)\\
\leq& \, C_1(\psi_0,\psi_1,\psi_2;\tau)
+C_{H^1, L^6}^\Omega \usmall{\|\nabla\gamma\|_{L^\infty L^3}} \|\nabla z^n_t\|_{L_t^2 L^2}^2 
+\usmall{\tfrac{\ul{\gamma}}{16}}\|\nabla z^n_t\|_{L_t^2 L^2}^2\\
&+\tfrac{8}{\ul{\gamma}}\Bigl(
C_{(-\Delta)^{-1}}^\Omega C_{H^2, W^{1,6}}^\Omega
\|-\Delta z^n\|_{L_t^2 L^2}
\tfrac{c^2}{b}\|\gamma\|_{L^\infty L^3}\\
&+C_{(-\Delta)^{-1}}^\Omega C_{H^2, L^\infty}^\Omega\|-\Delta z^n\|_{L_t^2 L^2} \tfrac{c^2}{b}\usmall{\|\nabla\gamma\|_{L^\infty L^2}} 
+\|\nabla f\|_{L^2 L^2} \\
&+\tfrac{c^4}{b^2} (\|\gamma\|_{L^\infty L^2}^2 + C_{PF}\|\nabla \gamma\|_{L^\infty L^2})
C_{(-\Delta)^{-1}}^\Omega C_{H^2, L^\infty}^\Omega 
\|\nabla(-\Delta) \psi^n\|_{L_t^2 L^2}
\\
&+ (C_{tr}^\Omega)^2 C_{PF}^\Omega \|f\|_{L^2 H^1} \Bigr)^2,
\end{aligned}
\end{equation}
since for the Galerkin discretization by eigenfunctions of the Laplacian, we have $(-\Delta) \psi^n \in H_0^1(\Omega)\cap H^2(\Omega)$ for smooth $\Omega$. 
All terms on the right-hand side except for 
\begin{equation}\label{rhs1til}
\begin{aligned}
\widetilde{\textbf{\mbox{rhs}}}_1:=& \, \begin{multlined}[t] C_1(\psi_0,\psi_1,\psi_2;\tau)+\tfrac{16}{\ul{\gamma}}\left( (1+(C_{tr}^\Omega)^2 C_{PF}^\Omega)\|f\|_{L^2 H^1} \right.\\
+C_{(-\Delta)^{-1}}^\Omega C_{H^2, W^{1,6}}^\Omega
\|-\Delta z^n\|_{L_t^2 L^2}\,
\tfrac{c^2}{b}\|\gamma\|_{L^\infty L^3}\\
+\tfrac{c^4}{b^2} \, \|\gamma\|_{L^\infty L^2}^2 \, C_{(-\Delta)^{-1}}^\Omega C_{H^2, L^\infty}^\Omega 
\|\nabla(-\Delta) \psi^n\|_{L_t^2 L^2}
)^2 \end{multlined}
\end{aligned}
\end{equation}
can be absorbed into the left-hand side \eqref{lhs} by making $\|\nabla\gamma\|_{L^\infty L^3}$ small. The right-hand side in \eqref{enid2_Kuz} can be estimated as follows
\begin{equation*}
\begin{aligned}
&\textbf{rhs}_2(t)\\
\leq& \, C_2(\psi_0,\psi_1,\psi_2;\tau)
+ \tau\|\nabla(-\Delta) \psi^n\|_{L_t^\infty L^2} \|\nabla z^n_t-\tfrac{c^2}{b}\nabla z^n+\tfrac{c^4}{b^2}\nabla \psi^n  \|_{L_t^\infty L^2}\\
&+\tfrac{\tau}{2}\|-\Delta z^n+\tfrac{c^2}{b}\Delta\psi^n  \|_{L_t^\infty L^2}^2\\
&+ \|-\Delta z^n+\tfrac{c^2}{b}\Delta\psi^n \|_{L_t^\infty L^2}
\Bigl(
\|-\Delta z^n+\tfrac{c^2}{b}\Delta\psi^n\|_{L_t^\infty L^2}\|\alpha\|_{L^1 L^\infty}\\
&+2\|\nabla z^n-\tfrac{c^2}{b}\nabla\psi^n\|_{L_t^\infty L^6}\|\nabla\alpha\|_{L^1 L^3}
+\|z^n-\tfrac{c^2}{b}\psi^n \|_{L_t^\infty L^\infty}\|-\Delta\alpha\|_{L^1 L^2}
\Bigr)\\
&+\|\nabla(-\Delta)\psi^n\|_{L_t^\infty L^2} \Bigl(
\|\nabla f\|_{L^1 L^2}
+\|\nabla z^n-\tfrac{c^2}{b}\nabla\psi^n \|_{L_t^\infty L^6}
(\|\alpha\|_{L^\infty L^3}+\|\alpha_t\|_{L^1 L^3}\\
&+\|z^n-\tfrac{c^2}{b}\psi^n \|_{L_t^\infty L^\infty}
\left(\|\nabla\alpha\|_{L^\infty L^2}+\|\nabla\alpha_t\|_{L^1 L^2}
\right) 
+(C_{tr}^\Omega)^2 C_{PF}^\Omega \|f\|_{L^1 H^1}
\Bigr),
\end{aligned}
\end{equation*}
where we have estimated the boundary term by means of the trace theorem
\[
\begin{aligned}
-\int_0^t \langle \nu\cdot \nabla (-\Delta)\psi^n, f\rangle_{H^{-1/2},H^{1/2}} \, \textup{d}s
\leq& \, (C_{tr}^\Omega)^2 \|(-\Delta)\psi^n\|_{L_t^\infty H^1} \|f\|_{L^1 H^1} \\[1mm]
\leq& \, (C_{tr}^\Omega)^2 C_{PF}^\Omega \|\nabla(-\Delta)\psi^n\|_{L_t^\infty L^2} \|f\|_{L^1 H^1}\,.
\end{aligned}
\]
We further have for $\lambda>0$ by Young's inequality that
\begin{equation}\label{estrhs2}
\begin{aligned}
&\lambda \cdot \textbf{rhs}_2(t) \,\\
\leq& \, \lambda \, C_2(\psi_0,\psi_1,\psi_2;\tau)
+\usmall{\lambda\tfrac{b}{16}}\|\nabla(-\Delta) \psi^n\|_{L_t^\infty L^2}^2\\
&+\usmall{\lambda\tfrac{4\tau}{b}}\, \tau \|\nabla z^n_t-\tfrac{c^2}{b}\nabla z^n+\tfrac{c^4}{b^2}\nabla \psi^n\|_{L_t^\infty L^2}^2 
+\usmall{\lambda\tfrac{\tau}{2}}\|-\Delta z^n+\tfrac{c^2}{b}\Delta\psi^n \|_{L_t^\infty L^2}^2\\
&+\usmall{\lambda}\|-\Delta z^n+\tfrac{c^2}{b}\Delta\psi^n \|_{L_t^\infty L^2}^2 \Bigl(\|\alpha\|_{L^1 L^\infty}+2C_{(-\Delta)^{-1}}^\Omega C_{H^2, W^{1,6}}^\Omega\|\nabla\alpha\|_{L^1 L^3}\\
&+C_{(-\Delta)^{-1}}^\Omega C_{H^2, L^\infty}^\Omega\|-\Delta\alpha\|_{L^1 L^2}
\Bigr)+\usmall{\lambda\tfrac{b}{16}}\|\nabla(-\Delta)\psi^n\|_{L_t^\infty L^2}^2 \\
&+\lambda\tfrac{8}{b}\left(1+\left(C_{tr}^\Omega \right)^2 C_{PF}^\Omega \right)^2\|f\|_{L^1 H^1}^2\\
&+\usmall{\lambda}\|-\Delta z^n-\tfrac{c^2}{b}(-\Delta)\psi^n\|_{L_t^\infty L^2}^2
  \tfrac{8}{b}\left(C_{(-\Delta)^{-1}}^\Omega \right)^2 \Bigl( C_{H^2, W^{1,6}}^\Omega (\|\alpha\|_{L^\infty L^3}\\ &+\|\alpha_t\|_{L^1 L^3})
+C_{H^2, L^\infty}^\Omega (\|\nabla\alpha\|_{L^\infty L^2}+\|\nabla\alpha_t\|_{L^1 L^2})\Bigr)^2,
\end{aligned}
\end{equation}
where by making $\tau$ and $\lambda$ small, all terms except for those containing the initial data and the inhomogeneity, 
\begin{equation}\label{rhs2til}
\begin{aligned}
&\lambda \cdot \widetilde{\textbf{rhs}}_2:=
\lambda C_2(\psi_0,\psi_1,\psi_2;\tau)+\lambda\tfrac{8}{b}\left(1+(C_{tr}^\Omega)^2 C_{PF}^\Omega \right)^2\|f\|_{L^1 H^1}^2
\end{aligned}
\end{equation} 
can be absorbed into the left-hand side given in \eqref{lhs}.
   
We now combine the energy estimates obtained from \eqref{enid1_Kuz}, and $\lambda$ times \eqref{enid2_Kuz} with a small constant $\lambda>0$ and absorb the indicated terms from the right-hand side estimates \eqref{estrhs1_Kuz}, \eqref{estrhs2} into the left-hand side so that only $\widetilde{\textbf{rhs}}_1$ and $\lambda \widetilde{\textbf{rhs}}_2$ remain on the right-hand side, cf. \eqref{lhs}, \eqref{rhs1til}, \eqref{rhs2til}. 
Therewith, we end up with an inequality of the form
\begin{align} \label{eta_estimate}
\eta(t)\leq C\left(\int_0^t \eta(s)\,\textup{d}s  +\|\psi_0\|^2_{H^3}+\|\psi_1\|^2_{H^2}+\tau\|\psi_2\|^2_{H^1}+ \|f\|_{L^2 H^1}^2\right),
\end{align}
for 
\begin{equation*}
\begin{aligned}
\eta(t)=& \,\begin{multlined}[t] \tfrac12\left(\tau \|\nabla z^n_t\|^2_{L^\infty(0,t;L^2)}+\ul{\gamma} \|\nabla z^n_t\|^2_{L^2(0,t;L^2)}+ b \|-\Delta z^n\|^2_{L^\infty(0,t;L^2)}\right)\\
+\lambda \left(2c^2\|\nabla(-\Delta) \psi^n\|^2_{L^2(0,t;L^2))}
+ b \|\nabla(-\Delta)\psi^n\|^2_{L^\infty(0,t;L^2)}\right), \end{multlined}
\end{aligned}
\end{equation*}
to which we employ Gronwall's lemma. 

To obtain a uniform bound on the full $H^3(\Omega)$ norm of $\psi^n$, we combine the $|\nabla(-\Delta)\psi^n|_{L^2}$ term with 
	$|-\Delta z^n|_{L^2}$ and the fact that $$\psi^n(x,t) =  e^{-(c^2/b)t} \psi_0(x) +\int_0^t e^{-(c^2/b)(t-s)} z^n(x,s)$$
	for $t \in (0, T)$. In this way, we have
	\[
	\begin{aligned}
	|\psi^n(t)|_{H^3(\Omega)}
	\leq& \, C_{(-\Delta)^{-1}}^\Omega \Bigl(|\nabla(-\Delta)\psi^n(t)|_{L^2}+|(-\Delta)\psi^n(t)|_{L^2}\Bigr)\\
	\leq& \,\begin{multlined}[t] C_{(-\Delta)^{-1}}^\Omega \Bigl(|\nabla(-\Delta)\psi^n(t)|_{L^2}+|e^{-(c^2/b)t}(-\Delta)\psi_0|_{L^2}\\
	+\left|\int_0^t e^{-(c^2/b)(t-s)}(-\Delta)z(s)\, \textup{d} s \right|_{L^2}\Bigr) \end{multlined}
	\end{aligned}
	\] 
	for a.e. $t \in (0,T)$. Altogether, we get the estimate
	\begin{equation}\label{reg_est_Kuz_n}
	\begin{aligned}
	& \begin{multlined}[t]\|\psi^n\|_{L^\infty H^3}^2+\|\psi^n_t+\tfrac{c^2}{b}\psi^n \|_{L^\infty H^2 }^2
    +\|\psi_{tt}+\tfrac{c^2}{b}\psi^n_t \|_{L^2 H_0^1 }^2 \\	
    +\tau  \|\psi^n_{tt}+\tfrac{c^2}{b}\psi^n_t\|_{L^\infty H_0^1 }^2 \end{multlined}\\
	\leq& \, C(\gamma,T) \left(|\psi_0|_{H^3}^2+|\psi_1|_{H^2}^2+\tau|\psi_2|_{H^1}^2+\|f\|_{L^2 H^1}^2\right),
	\end{aligned}
	\end{equation}
with a constant $C(\gamma,T)>0$ independent of $\tau$, provided $\|\nabla\gamma\|_{L^\infty L^3}$ is sufficiently small. \\[2mm]
\noindent \textbf{Step 3: Passing to the limit.} On account of estimate \eqref{reg_est_Kuz_n} and the Banach-Alaoglu theorem, we know that there exists a subsequence, denoted again by $\{\psi^n\}_{n \in \mathbb{N}}$, and a function $\tilde{\psi}$ such that
\begin{alignat*}{4}
\psi^n  &\relbar\joinrel\rightharpoonup \tilde{\psi} &&\text{ weakly-$\star$}  &&\text{ in } &&L^\infty(0,T;H_0^1(\Omega)\cap H^3(\Omega)),  \\
\psi_t^n &\relbar\joinrel\rightharpoonup \tilde{\psi}_t &&\text{ weakly-$\star$} &&\text{ in } &&L^\infty(0,T;H^1_0(\Omega)\cap H^2(\Omega)),\\
\psi_t^n &\relbar\joinrel\rightharpoonup \tilde{\psi}_t &&\text{ weakly} &&\text{ in } &&L^2(0,T;H_0^1(\Omega)),\\
\psi_{tt}^n &\relbar\joinrel\rightharpoonup \tilde{\psi}_{tt} &&\text{ weakly} &&\text{ in } &&L^2(0,T;H_0^1(\Omega)), \\
\psi_{tt}^n &\relbar\joinrel\rightharpoonup \tilde{\psi}_{tt} &&\text{ weakly-$\star$}  &&\text{ in } &&L^\infty(0,T;H_0^1(\Omega)).
\end{alignat*} 
By uniqueness of limits $\tilde{\psi}$ has to coincide with the solution $\psi$ according to Theorem \ref{th:wellposedness_lin}, which thus satisfies \eqref{reg_est_Kuz}.
\end{proof}

\section{Well-posedness for the nonlinear Kuznetsov-type wave equation for $\tau>0$ sufficiently small}\label{sec:wellposed_Kuz}

We next intend to employ the Banach fixed-point theorem to prove well-posedness for equation \eqref{KuznetsovMC}. To this end, we introduce the operator $\mathcal{T}$ that maps $\phi$ to a solution $\psi$ of 
\[
\tau\psi_{ttt}+(1-k\phi_t)\psi_{tt}-c^2\Delta \psi - b\Delta \psi_t
= 2\nabla\phi\cdot\nabla\phi_t\,,
\]
on some ball 
\begin{equation}\label{Brho_K}
\begin{aligned}
B_\rho^{X^K}=\Bigl\{\psi\in X^K \, :& \ \psi(0)=\psi_0\,, \ \psi_t(0)=\psi_1\,, \ \psi_{tt}(0)=\psi_2\,, \mbox{ and }\\ 
&\|\psi\|_{X^K}^2:=\tau \|z_t\|^2_{L^\infty H^1} + \|z_t\|^2_{L^2 H^1}\\
&\hspace{1.7cm}+\|z\|^2_{L^\infty H^2}+\|\psi\|^2_{L^\infty H^3}\leq \rho^2, \\
&\mbox{ for }z=\psi_t+\frac{c^2}{b}\psi \Bigr\}
\end{aligned}
\end{equation}
in the space $X^K$, defined in \eqref{XWXKspaces}. Thus, for establishing $\mathcal{T}$ as a self-mapping on $B_\rho^{X^K}$, it is crucial to prove that
$\alpha=1-k\phi_t$ and $f=2\nabla\phi\cdot\nabla\phi_t$ are in 
$W^{1,1}(0,T;H^1(\Omega))\cap L^\infty(0,T;W^{1,3}(\Omega)\cap L^1(0,T;H^2(\Omega)$
and $L^2(0,T;H^1(\Omega))$, respectively, and that the derivatives of $\alpha$ are small when $\phi\in B_\rho^{X^K}$. \\
\indent Concerning non-degeneracy, we assume that $\tau\in(0,\bar{\tau}]$ with $\bar{\tau}<\tfrac{b}{c^2}$ so that 
\begin{equation}\label{taubar}
\gamma^*:= 1-\bar{\tau}\tfrac{c^2}{b} >0\,.
\end{equation}
Therefore, keeping 
\begin{equation}\label{m}
\|\gamma-\gamma^*\|_{L^\infty L^\infty}
=\|\alpha-1\|_{L^\infty L^\infty}
=k\|\phi_t\|_{L^\infty L^\infty}\leq m
\end{equation} 
with $m<\gamma^*$ allows to choose
\begin{equation}\label{ulgamma}
\ul{\gamma}:=1-\bar{\tau}\tfrac{c^2}{b} -m>0\,, \quad \ul{\alpha}:=1-m>0
\end{equation}
in \eqref{eq:alphagamma_pos} independently of $\tau$, which will also be important for the considerations in Section \ref{sec:limits}.
Thus we also need to verify that \eqref{m} follows from $\phi\in B_\rho^{X^K}$.

To additionally obtain contractivity, based on the fact that the difference $\hat{\psi}=\psi_1-\psi_2=\mathcal{T}(\phi_1)-\mathcal{T}(\phi_2)$ solves
\begin{equation}\label{hat}
\begin{aligned}
&\tau\hat{\psi}_{ttt}+(1-k\phi_{1\,t})\hat{\psi}_{tt}-c^2\Delta \hat{\psi} - b\Delta \hat{\psi}_t \\
=& \, k\hat{\phi}_t \psi_{2\,tt}+2\nabla\hat{\phi}\cdot\nabla\phi_{1\,t}+2\nabla\phi_2\cdot\nabla\hat{\phi}_t\,,
\end{aligned}
\end{equation}
with homogeneous initial and boundary conditions (where $\hat{\phi}=\phi_1-\phi_1$),
we need to prove that
$\alpha_1=1-k\phi_{1\,t}$ and $f_2=k\hat{\phi}_t \psi_{2\,tt}+2\nabla\hat{\phi}\cdot\nabla\phi_{1\,t}+2\nabla\phi_2\cdot\nabla\hat{\phi}_t$ are in 
$W^{1,1}(0,T;H^1(\Omega))\cap L^\infty(0,T;W^{1,3}(\Omega)\cap L^1(0,T;H^2(\Omega)$
and $L^2(0,T;H^1(\Omega))$,  and that the derivatives of $\alpha_1$ are small, provided 
$\phi_1,\phi_2\in B_\rho^{X^K}$. Moreover, $\|f_2\|_{L^2 H^1}$ needs to be estimated by a multiple of $\|\hat{\phi}\|_{X^K}$ with a small factor.

\begin{theorem} \label{th:wellposedness_Kuz}
	Let $c^2$, $b$, $T>0$, $k\in\mathbb{R}$. Then there exist $\bar{\tau}$, $\rho>0$ , $\rho_0>0$ such that for all $(\psi_0,\psi_1,\psi_2)\in X^K_0
=H_0^1(\Omega)\cap H^3(\Omega)\times H_0^1(\Omega)\cap H^2(\Omega)\times H_0^1(\Omega)$ 
satisfying 
\begin{equation}\label{smallness_init}
\|\psi_0\|_{H^3(\Omega)}^2+\|\psi_1\|_{H^2(\Omega)}^2+\tau\|\psi_2\|_{H^1(\Omega)}^2\leq\rho_0^2\,,
\end{equation} 
 and all $\tau\in(0,\bar{\tau})$, there exists a unique solution $\psi\in X^K$ of 
	\begin{equation} \label{ibvp_Kuznetsov_MC}
	\begin{aligned}
	\begin{cases}
	\tau\psi_{ttt}+\psi_{tt}-c^2\Delta \psi - b\Delta \psi_t = \left(
\tfrac{k}{2}
(\psi_t)^2+|\nabla\psi|^2\right)_t
 \quad \mbox{ in }\Omega\times(0,T), \\[1mm]
	\psi=0 \quad \mbox{ on } \partial \Omega\times(0,T),\\[1mm]
	(\psi, \psi_t, \psi_{tt})=(\psi_0, \psi_1, \psi_2) \quad \mbox{ in }\Omega\times \{0\},
	\end{cases}
	\end{aligned}
	\end{equation}
which satisfies the estimate
\[
\tau \|z_t\|^2_{L^\infty H^1} + \|z_t\|^2_{L^2 H^1}+\|z\|^2_{L^\infty H^2}+\|\psi\|^2_{L^\infty H^3}\leq \rho^2.
\]
\end{theorem}
\begin{proof}
We first prove that $\mathcal{T}$ is a self-mapping on $B_\rho^{X^K}$.
For this purpose, we estimate $\alpha$ and $f$, assuming that $\phi\in B_\rho^{X^K}$ and abbreviating $w=\phi_t+\frac{c^2}{b}\phi$:
\begin{equation}\label{estalpha}
\begin{aligned} 
\|\nabla\alpha_t\|_{L^1 L^2}
&=|k|\,\|\nabla w_t-\tfrac{c^2}{b}\nabla w+\tfrac{c^4}{b^2}\nabla\phi\|_{L^1 L^2}\\
&\leq |k|\, \sqrt{T}(1+\tfrac{c^2}{b}\sqrt{T}+\tfrac{c^4}{b^2}) \rho,\\
\|\nabla\alpha\|_{L^\infty L^3}
&=|k|\,\|\nabla w-\tfrac{c^2}{b}\nabla \phi\|_{L^\infty L^3}\\
&\leq |k|\, C_{H^2\to W^{1,3}}^\Omega (1+\tfrac{c^2}{b}) \rho,\\
\|-\Delta\alpha\|_{L^1 L^2}
&=|k|\,\|-\Delta w+\tfrac{c^2}{b}\Delta \phi\|_{L^1 L^2}\\
&\leq |k|\, \sqrt{T}(1+\tfrac{c^2}{b}) \rho,\\
\|\gamma-\gamma^*\|_{L^\infty L^\infty}
&=\|\alpha-1\|_{L^\infty L^\infty}
=|k|\,\|w-\tfrac{c^2}{b}\phi\|_{L^\infty L^\infty}\\
&\leq |k|\, C_{H^2\to L^\infty}^\Omega (1+\tfrac{c^2}{b}) \rho.
\end{aligned}
\end{equation}
Moreover, we find that  
\[
\begin{aligned} 
\|\nabla f\|_{L^2 L^2}
=& \, 2 \left\|\nabla^2\phi \nabla\phi_t+\nabla^2\phi_t \nabla\phi \right\|_{L^2 L^2}\\[1mm]
\leq& \, 2 \begin{multlined}[t] \Bigl(\|\nabla^2\phi\|_{L^\infty L^6} 
\|\nabla w-\tfrac{c^2}{b}\nabla\phi \|_{L^2 L^3}\\
+\|\nabla^2 w-\tfrac{c^2}{b}\nabla^2\phi \|_{L^2 L^2} \|\nabla\phi\|_{L^\infty L^\infty}\Bigr) \end{multlined} \\
\leq& \, 2(C_{H^3\to W^{2,6}}^\Omega C_{H^2, W^{1,3}}^\Omega +C_{H^3, W^{1,\infty}}^\Omega \left(1+\tfrac{c^2}{b} \right) \rho^2 =:C_f \rho^2
\,.
\end{aligned}
\]
Therefore, energy estimate \eqref{reg_est_Kuz} yields
\[
\|\psi\|_{X^K}^2\leq C(\gamma,T)\Bigl(C_f^2\rho^4+\|\psi_0\|_{H^3(\Omega)}^2+\|\psi_1\|_{H^2(\Omega)}^2+\tau\|\psi_2\|_{H^1(\Omega)}^2\Bigr)
\leq\rho^2\,,
\]
provided that the initial data are small in the sense of \eqref{smallness_init} with 
\begin{equation}\label{rhorho0}
\begin{aligned}
&\rho_0<\frac{1}{4C(\gamma,T) C_f}\,, \\
&\rho\leq\min\left\{\frac{1+\sqrt{1-4C(\gamma,T)C_f\rho_0}}{2C(\gamma,T)C_f},\,\frac{m}{kC_{H^2\to L^\infty}^\Omega(1+c^2/b)}\right\}\,,
\end{aligned}
\end{equation}
which implies that $\mathcal{T}$ maps $B_\rho^{X^K}$ into itself. 

For proving contractivity of $\mathcal{T}$, we estimate $\alpha_1$ analogously to \eqref{estalpha},
and with abbreviations $\hat{w}:=\hat{\phi}_t+\frac{c^2}{b}\hat{\phi}$, $w_i:=\phi_{i\,t}+\frac{c^2}{b}\phi_i$, $z_i:=\psi_{i\,t}+\frac{c^2}{b}\psi_i$, $i\in\{1,2\}$. We have
\[
\begin{aligned} 
&\|\nabla f_2\|_{L^2 L^2} \\
\leq& \, |k|\Bigl(\|\nabla \hat{w}-\tfrac{c^2}{b}\nabla \hat{\phi}\|_{L^\infty L^3}  \|z_{2\,t}-\tfrac{c^2}{b}z_2+\tfrac{c^4}{b^2}\psi_2 \|_{L^2 L^6} \\
& + \|\hat{w}-\frac{c^2}{b}\hat{\phi}\|_{L^\infty L^\infty}  \|\nabla z_{2\,t}-\tfrac{c^2}{b}\nabla z_2+\tfrac{c^4}{b^2}\nabla\psi_2\|_{L^2 L^2} \Bigr)
\\
&+ 2  \Bigl(\|\nabla^2\hat{\phi}\|_{L^\infty L^6} 
\|\nabla w_1-\tfrac{c^2}{b}\nabla\phi_1\|_{L^2 L^3}
+\|\nabla^2 \hat{w}-\tfrac{c^2}{b}\nabla^2\hat{\phi}\|_{L^2 L^2} 
\|\nabla\phi_1\|_{L^\infty L^\infty}\Bigr)\\
&+ 2  \Bigl(\|\nabla^2 \phi_2\|_{L^\infty L^6} 
\|\nabla \hat{w}-\tfrac{c^2}{b}\nabla\hat{\phi}\|_{L^2 L^3}
+\|\nabla^2 w_2-\tfrac{c^2}{b}\nabla^2\phi_2\|_{L^2 L^2} 
\|\nabla\hat{\phi}\|_{L^\infty L^\infty}\Bigr).
\end{aligned}
\] 
From here it folllows that
\[
\begin{aligned} 
&\|\nabla f_2\|_{L^2 L^2} \\
\leq& \,\begin{multlined}[t] \left (|k|\|\psi\|_{X^K}+2\|\phi_1\|_{X^K}+2\|\phi_2\|_{X^K}\right)\left(C_{H^2\to W^{1,3}}^\Omega C_{H^1\to L^6}^\Omega + C_{H^2\to L^\infty}^\Omega \right)\\
 \times \left(1+\tfrac{c^2}{b}+\tfrac{c^4}{b^2}\right) \|\hat{\phi}\|_{X^K} \end{multlined}\\
\leq& \, \hat{C}_f \rho\|\hat{\phi}\|_{X^K}\,.
\end{aligned}
\] 
By applying estimate \eqref{reg_est_Kuz} to equation \eqref{hat} with homogeneous initial conditions, we obtain
\[
\|\hat{\psi}\|_{X^K}\leq \sqrt{C(\gamma,T)}\hat{C}_f\rho \, \|\hat{\phi}\|_{X^K}\,,
\]
which after possibly decreasing $\rho$ yields contractivity.

Since $B_\rho^{X^K}$ is closed, we can make use of Banach's contraction principle to conclude existence and uniqueness of a solution.
\end{proof}

\begin{remark}[On global well-posedness]
Note that $C(\gamma,T)$ in \eqref{reg_est_Kuz} depends on the final time due to the use of Gronwall's inequality. We do not expect that global in time well-posedness can be proven in the nonlinear case due to the fact that we must deal with a quadratic nonlinearity and only have weak damping in the equation for $z$.
\end{remark}

\section{Singular limit for vanishing relaxation time}\label{sec:limits}
We next focus on proving a limiting result for equations \eqref{KuznetsovMC} and \eqref{WesterveltMC} as $\tau\to0$.
Recall that $b=\delta+\tau c^2$ and that the norms on the spaces $X^W$ and $X^K$, defined in \eqref{XWXKspaces}, depend on $\tau$, 
whereas the radius $\rho$ of the balls \eqref{Brho_W} and \eqref{Brho_K} is independent of $\tau$. \\
\indent As already stated in the notational preliminaries \eqref{XWXKnorms}, we denote by $\|\cdot\|_{\bar{X}^W}$ and $\|\cdot\|_{\bar{X}^K}$ the respective $\tau$-independent part of the norms defined in \eqref{Brho_W} and \eqref{Brho_K}, 
\begin{equation*}
\begin{aligned}
&\|\psi\|_{\bar{X}^W}^2:=\|\psi_{tt}\|^2_{L^2 H^1}+\|\psi\|^2_{W^{1,\infty} H^2}\,, \\
&\|\psi\|_{\bar{X}^K}^2:=\|\psi_{tt}+\tfrac{c^2}{b}\psi_t\|^2_{L^2 H^1}+\|\psi_t+\tfrac{c^2}{b}\psi\|^2_{L^\infty H^2}+\|\psi\|^2_{L^\infty H^3}.
\end{aligned}
\end{equation*}
Moreover, we recall the spaces for the initial data  
\[
\begin{aligned}
&X_0^W:=H_0^1(\Omega)\cap H^2(\Omega)\times H_0^1(\Omega)\cap H^2(\Omega) \times H_0^1(\Omega),\\
&X_0^K:=H_0^1(\Omega)\cap H^3(\Omega)\times H_0^1(\Omega)\cap H^2(\Omega) \times H_0^1(\Omega),
\end{aligned}
\]
the only difference being in the regularity of $\psi^0$. Therewith, we can formulate a limiting result for \eqref{KuznetsovMC} and \eqref{WesterveltMC}.

\begin{theorem} \label{th:limits}
	Let $c^2$, $b$, $T>0$, and $k \in \mathbb{R}$. Then there exist $\bar{\tau}$, $\rho_0>0$ such that for all $(\psi_0,\psi_1,\psi_2)\in X_0^W$, the family $(\psi^\tau)_{\tau\in(0,\bar{\tau})}$ of solutions to \eqref{ibvp_Westervelt_MC} 
	according to Theorem \ref{th:wellposedness_Wes} 
	converges weakly* in $\bar{X}^W$ to a solution $\bar{\psi}\in \bar{X}^W$ of \eqref{Westervelt} with homogeneous Dirichlet boundary conditions \eqref{Dirichlet} and initial conditions $\bar{\psi}(0)=\psi_0$, $\bar{\psi}_t(0)=\psi_1$.
	
	The statement remains valid with the equations \eqref{ibvp_Westervelt_MC}, \eqref{Westervelt}, the spaces $X_0^W$, $\bar{X}^W$ and Theorem \ref{th:wellposedness_Wes} replaced by the equations \eqref{ibvp_Kuznetsov_MC}, \eqref{Kuznetsov}, the spaces $X_0^K$, $\bar{X}^K$ and Theorem \ref{th:wellposedness_Kuz}, respectively.
\end{theorem}

\begin{proof}
	From the energy estimates in Theorems \ref{th:wellposedness_Wes} (or \ref{th:wellposedness_Kuz}), we have uniform boundedness of $(\psi^\tau)_{\tau\in(0,\bar{\tau})}$ in $\bar{X}^W$ (or in $\bar{X}^K$) and therefore existence of a weakly* $\bar{X}^W$ (or $\bar{X}^K$) convergent sequence $(\psi^\ell)_{\ell\in\mathbb{N}}$ with $\tau_\ell\searrow0$. By compactness of embeddings, this sequence also converges strongly in  $C^1(0,T; L^4(\Omega))\cap C(0,T; W^{1,4}(\Omega))$. Its limit $\bar{\psi}$ therefore lies in $\bar{X}^W$ (or in $\bar{X}^K$) and satisfies the initial conditions $\bar{\psi}(0)=\psi_0$, $\bar{\psi}_t(0)=\psi_1$.    
	
	To prove that $\bar{\psi}$ also satisfies the respective PDEs, we test with arbitrary functions $v \in C_0^\infty(0,T;C_0^\infty(\Omega))$ and invoke the Fundamental Lemma of Calculus of Variations.
	To this end, we introduce an abbreviation for the nonlinear term in the respective equations, namely
	\[
	\begin{aligned}
	\mathcal{N}(\psi)=& \, \begin{cases} 
	\frac{\beta_a}{c^2}(\psi_t)^2 \mbox{ for \eqref{Westervelt}, \eqref{WesterveltMC}}\\
	\frac{1}{c^2}\frac{B}{2A}(\psi_t)^2+|\nabla\psi|^2 \mbox{ for \eqref{Kuznetsov}, \eqref{KuznetsovMC},}
	\end{cases}\\
	=&\, k (\psi_t)^2 + \sigma |\nabla\psi|^2 \ \mbox{ with } \sigma = \begin{cases} 
	0 \mbox{ for \eqref{Westervelt}, \eqref{WesterveltMC}}\\
	1 \mbox{ for \eqref{Kuznetsov}, \eqref{KuznetsovMC}}.
	\end{cases}
	\end{aligned}
	\]
	\textcolor{mygreen}{Now} the $\tau$-dependent and the limiting equation can be rewritten in both Westervelt and Kuznetsov cases as 
	\begin{equation}\label{tauPDE}
	\tau\psi_{ttt}+\psi_{tt}-c^2\Delta \psi - (\delta+\tau c^2)\Delta \psi_t- (\mathcal{N}(\psi))_t=0
	\end{equation}
	and 
	\begin{equation}\label{limitPDE}
	\psi_{tt}-c^2\Delta \psi - \delta\Delta \psi_t- (\mathcal{N}(\psi))_t=0,
	\end{equation}
	respectively. \\
	\indent Note that by the regularity inherent in the spaces $X^W$, $X^K$ and $\bar{X}^W$, $\bar{X}^K$, 
cf. \eqref{XWXKspaces}, $\psi^\tau$ satisfies equation \eqref{tauPDE} in $L^2(0;T;L^2(\Omega))$. Inserting $\bar{\psi}$ into the left-hand side of \eqref{limitPDE} yields an $L^2(0;T;L^2(\Omega))$ function.
	Therewith, we get, for $\hat{\psi}_\ell:=\bar{\psi}-\psi^\ell$ and any $v\in C_0^\infty(0,T;C_0^\infty(\Omega))$ that
	\[
	\begin{aligned}
	&\int_0^T\int_\Omega \Bigl(\bar{\psi}_{tt}-c^2\Delta \bar{\psi} - \delta\Delta \bar{\psi}_t- \mathcal{N}(\bar{\psi})_t\Bigr)\, v \, \textup{d}x\, \textup{d}t\\
	=& \,\int_0^T\int_\Omega \Bigl(\hat{\psi}_{\ell\,tt}-c^2\Delta \hat{\psi}_\ell - \delta\Delta \hat{\psi}_{\ell\,t}- (\mathcal{N}(\bar{\psi})_t-\mathcal{N}(\psi^\ell))_t
	-\tau_\ell\psi^\ell_{ttt} -\tau_\ell c^2\Delta \psi_{\ell\,t}\Bigr)\, v \, \textup{d}x\, \textup{d}t\\
	=& \, \mbox{I} - \mbox{II} - \mbox{III}\,.
	\end{aligned}
	\]
	Above, we have that
	\[
	\begin{aligned}
	\mbox{I}=
	\int_0^T\int_\Omega \Bigl(\hat{\psi}_{\ell\,tt}-c^2\Delta \hat{\psi}_\ell - \delta\Delta \hat{\psi}_{\ell\,t}\Bigr)\, v \, \textup{d}x\, \textup{d}t \ \to 0 \ \mbox{ as }\ \ell\to\infty
	\end{aligned}
	\]
	due to the weak* convergence to zero of $\hat{\psi}_\ell$ in $\bar{X}^W$ (or $\bar{X}^K$). Moreover,
	\[
	\begin{aligned}
	\mbox{II}=& \,
	\int_0^T\int_\Omega \Bigl((\mathcal{N}(\bar{\psi})-\mathcal{N}(\psi^\ell))\Bigr)_t\, v \, \textup{d}x\, \textup{d}t \\
	=&\, -\int_0^T\int_\Omega \Bigl((\mathcal{N}(\bar{\psi})-\mathcal{N}(\psi^\ell))\Bigr)\, v_t \, \textup{d}x\, \textup{d}t\\
	=&\, -\int_0^T\int_\Omega \Bigl(k(\bar{\psi}_t+\psi^\ell_t)\hat{\psi}_{\ell\,t}+\sigma (\nabla\bar{\psi}+\nabla\psi^\ell)\cdot\nabla\hat{\psi}_\ell\Bigr)\, v_t \, \textup{d}x\, \textup{d}t \ \to 0 \ \mbox{ as } \ \ell\to\infty
	\end{aligned}
	\]
	due to the boundedness of $(\psi_\ell)_{\ell\in\mathbb{N}}$ in $\bar{X}^W$ (or $\bar{X}^K$) by $\rho$, and the strong convergence to zero of $\hat{\psi}_\ell$ in  $C^1(0,T; L^4(\Omega))\cap C(0,T; W^{1,4}(\Omega))$. Finally,
	\[
	\begin{aligned}
	\mbox{III}=& \,
	\tau_\ell \, \int_0^T\int_\Omega \Bigl(\psi^\ell_{ttt}+c^2\Delta \psi_{\ell\,t}\Bigr)\, v \, \textup{d}x\, \textup{d}t\\
	=&\, \tau_\ell\int_0^T\int_\Omega \Bigl(\psi^\ell_{tt}+c^2\Delta \psi_\ell\Bigr)\, v_t \, \textup{d}x\, \textup{d}t \ \to 0 \ \mbox{ as } \ \ell\to\infty
	\end{aligned}
	\]
	due to the boundedness of $(\psi_\ell)_{\ell\in\mathbb{N}}$ in $\bar{X}^W$ (or $\bar{X}^K$), and $\tau_\ell\to0$.
	
	A subsequence-subsequence argument, together with uniqueness of the solution to \eqref{limitPDE} according to results in, e.g., Refs.~\refcite{KL09Westervelt,KL12_Kuznetsov,MW11,MW13} yields convergence of the whole family $(\psi^\tau)_{\tau\in(0,\bar{\tau})}$.
\end{proof}

\begin{remark}[On compatibility conditions]
	Note that, in contrast to Ref.~\refcite{KT18_ModelsNlAcoustics}, no compatibility condition on $\psi_{2}$ is needed, since no continuity of the limit $\bar{\psi}_{tt}$ with respect to time arises and in the used energy estimates the $\psi_2$ term vanishes as $\tau\to0$.
\end{remark}
\begin{remark}[On strong convergence]
We could look directly at the equation solved by the difference $ \hat{\psi} = \psi^\tau-\bar{\psi}$ of solutions to the JMGT and the Westervelt equation:
	\begin{equation}
	\begin{aligned}
	\hat{\psi}_{tt}-c^2\Delta\hat{\psi}-\delta\Delta\hat{\psi}_t-k((\psi^\tau+\bar{\psi})\hat{\psi})_t
	= -\tau \psi^\tau_{ttt}-\tau c^2\Delta\psi^\tau_t.
	\end{aligned}
	\end{equation}
	However, showing that $\hat{\psi}$ tends to zero as $\tau \rightarrow 0$ appears to be beyond our theoretical reach, although suggested by the numerical results in Section~\ref{sec:NumEx}. A particular challenge is to estimate the first term  $\tau \psi^\tau_{ttt}$: we only know this it is bounded by $\rho$ in $L^2(0,T; L^2(\Omega))$ according to estimate \eqref{est_Westervelt}, but not that it tends to zero. An analogous argument can be made for the Kuznetsov-type JMGT equation. 
\end{remark}
\subsection{Comparison to the regularity results in the literature}
We note that Theorem~\ref{th:limits} also contains a regularity result on the solutions $\bar{\psi}\in \bar{X}^W$ and $\bar{\psi}\in \bar{X}^K$ of the Westervelt \eqref{Westervelt} and the Kuznetsov \eqref{Kuznetsov} equations with homogeneous Dirichlet boundary conditions \eqref{Dirichlet} and initial conditions $\bar{\psi}(0)=\psi_0\in H_0^1(\Omega)\cap H^2(\Omega)$ or $H^3(\Omega)$, $\bar{\psi}_t(0)=\psi_1\in H_0^1(\Omega)$.
	
	By comparing this regularity with the regularity results on the Westervelt equation from Refs.~\refcite{KL09Westervelt,MW11} and with those for the Kuznetsov equation from Ref.~\refcite{KL12_Kuznetsov,MW13,MizohataUkai}, 
and noting that in Refs.~\refcite{KL09Westervelt,MW11,KL12_Kuznetsov,MW13},
$u$ is the acoustic pressure, i.e., related to $\psi$ by $u=\varrho_0\psi_t$,
we get
	\begin{itemize}
		\item Westervelt equation:
		\begin{itemize}
			\item 
			Ref.~\refcite{KL09Westervelt}:
			$(u_0,u_1)\in (H_0^1(\Omega)\cap H^2(\Omega))\times H_0^1(\Omega)$ 
\\[1mm] 
and additionally $(1-ku_0)^{-1}[ c^2 \Delta u_0 + b \Delta u_1 + k u_1^2]\in L^2(\Omega)$\\[1mm]  $\Rightarrow$ 
			$u\in C^2(0,T;L^2(\Omega))\cap H^2(0,T;H_0^1(\Omega))\cap C(0,T;H^2(\Omega))$;\\
			\item 
			Ref.~\refcite{MW11}:
			$(u_0,u_1)\in (H_0^1(\Omega)\cap H^2(\Omega))\times H_0^1(\Omega)$\\[1mm] $\Rightarrow$ 
			$u\in H^2(0,T;L^2(\Omega))\cap H^1(0,T;H_0^1(\Omega)\cap H^2(\Omega))$;\\
			\item here:
			$(\psi_0,\psi_1)\in \left(H_0^1(\Omega)\cap H^2(\Omega)\right)^2$ 
\\[1mm]  and additionally $(1-k\psi_0)^{-1}[ c^2 \Delta \psi_0 + b \Delta \psi_1 ]\in H_0^1(\Omega)$\\[1mm] $\Rightarrow$
			$u\in H^2(0,T;H_0^1(\Omega))\cap W^{1,\infty}(0,T;H^2(\Omega))$.\\
		\end{itemize}
		\item Kuznetsov equation:
		\begin{itemize}
			\item 
			Ref.~\refcite{MizohataUkai}: 
			$(\psi_0,\psi_1)\in H_0^1(\Omega)\cap H^3(\Omega)\times H_0^1(\Omega)\cap H^2(\Omega)$\\[1mm] $\Rightarrow$ 
			$\psi\in C^1(0,T;H_0^1(\Omega)\cap H^2(\Omega))\cap H^1(0,T;H^3(\Omega))$;\\
			\item 
			Ref.~\refcite{KL12_Kuznetsov}:
			$(u_0,u_1)\in (H_0^1(\Omega)\cap H^2(\Omega))\times H_0^1(\Omega)$ 
\\[1mm] and additionally $(1-ku_0)^{-1}[ c^2 \Delta u_0 + b \Delta u_1 + k u_1^2 + 2|\nabla u_0|^2]\in L^2(\Omega)$\\[1mm] $\Rightarrow$ 
			$u\in C^2(0,T;L^2(\Omega))\cap H^2(0,T;H_0^1(\Omega))\cap C(0,T;H^2(\Omega))$;\\
			\item 
			Ref.~\refcite{MW13}:
			$(u_0,u_1)\in (H_0^1(\Omega)\cap H^2(\Omega))\times H_0^1(\Omega)$\\[1mm] $\Rightarrow$ 
			$u\in H^{5/2}(0,T;H_0^1(\Omega))\cap H^2(0,T;H^2(\Omega))$;\\
			\item here:
			$(\psi_0,\psi_1)\in H_0^1(\Omega)\cap H^3(\Omega)\times H_0^1(\Omega)\cap H^2(\Omega)$ 
\\[1mm] and additionally $(1-k\psi_1)^{-1}[ c^2 \Delta \psi_0 + b \Delta \psi_1 + \nabla\psi_0\cdot\nabla\psi_1]\in H_0^1(\Omega)$\\[1mm] $\Rightarrow$
			$u\in H^2(0,T;H_0^1(\Omega))\cap W^{1,\infty}(0,T;H^2(\Omega))\cap L^\infty(0,T;H^3(\Omega))$.
		\end{itemize}
	\end{itemize}
	We point out that these works also contain results on global in time existence and exponential decay of solutions, as well as, in case of Refs.~\refcite{DekkersRozanova,MizohataUkai}, on the Cauchy problem, and, in case of Refs.~\refcite{MW11,MW13}, in general, non-Hilbert $L^p(\Omega)$ and $W^{s,p}(\Omega)$ spaces.  
Moreover, we wish to point to Ref.~\refcite{DoerflerGernerSchnaubelt16}, where local in time well-posedness of a class of quasilinear wave equations without strong damping was shown, which also comprises the Westervelt equation with $\delta=0$. It yields the regularity $u\in C^2(0,T;H^1(\Omega))\cap C^1(0,T;H^2(\Omega)) \cap C(0,T;H^3(\Omega))$ for initial data $(u_0,u_1)\in H_0^1(\Omega)\cap H^2(\Omega)\times H_0^1(\Omega)\cap H^3(\Omega)$ with $\Delta u_1\vert_{\partial\Omega}=0$.

\section{Numerical results} \label{sec:NumEx}
\indent As an illustration of our theoretical findings, we solve and compare numerically equations \eqref{Westervelt} and \eqref{WesterveltMC} in a one-dimensional channel geometry. For the medium, we choose water with parameters 
$$c=1500 \, \textup{m}/\textup{s}, \ \delta=6\cdot 10^{-9} \, \textup{m}^2/\textup{s}, \ \rho = 1000 \, \textup{kg}/\textup{m}^3, \ B/A=5;$$  cf. Chapter 5 in Ref.~\refcite{kaltenbacher2007numerical}. Recall that $b=\delta+\tau c^2$; the choice of the relaxation parameter $\tau$ is given below. Discretization in space is performed by employing B-splines as basis functions within the framework of Isogeometric Analysis (IGA); see~Refs.~\refcite{Cottrell,HughesBook}. For a detailed insight into the application of Isogeometric Analysis in nonlinear acoustics, we refer to Refs.~\refcite{Fritz,2017isogeometricPaper}. We use quadratic basis functions with the maximum $C^1$ global regularity and have $251$ degrees of freedom for the channel length $l=0.2 \,$m. The nonlinearities are resolved by a fixed-point iteration with the tolerance set to $\textup{TOL}=10^{-8}$. \\
\indent After discretizing in space, we end up with a semi-discrete matrix equation and proceed with a time-stepping scheme. For the Westervelt equation \eqref{WesterveltMC}, we employ the standard Newmark relations\cite{Newmark} for second-order equations:
\begin{equation} \label{Newmark_West}
\begin{aligned}
\underline{\psi}^{n+1}=& \, \underline{\psi}^n+\Delta t \dot{\underline{\psi}}^n+\frac{(\Delta t)^2}{2!} \left((1-2\beta)\ddot{\underline{\psi}}^n+2\beta \, \ddot{\underline{\psi}}^{n+1} \right),\\
\dot{\underline{\psi}}^{n+1}=& \, \dot{\underline{\psi}}^{n}+\Delta t\, \left((1-\gamma)\ddot{\underline{\psi}}^n+\gamma \, \ddot{\underline{\psi}}^{n+1} \right),
\end{aligned}
\end{equation}
realized through a predictor-corrector scheme analogously to Algorithm 1 in Ref.~\ \refcite{2017isogeometricPaper}. In \eqref{Newmark_West}, $\Delta t$ denotes the time step size. The vectors $\underline{\psi}^n$, $\underline{\dot{\psi}}^n$, and $\underline{\ddot{\psi}}^n$ denote the discrete acoustic potential, its first time derivative, and its second time derivative, respectively, at the time step $n$.\\
\indent For the Jordan--Moore--Gibson--Thompson equation with Westervelt-type nonlinearity \eqref{WesterveltMC}, we use an extension of the Newmark relations to third-order models similar to the one employed in Appendix B.2 of Ref.~\refcite{fathi2015time}:
\begin{equation} \label{Newmark_JMGT}
\begin{aligned}
\underline{\psi}^{n+1}=& \, \psi^n+\Delta t \dot{\underline{\psi}}^n+\frac{(\Delta t)^2}{2!}\ddot{\underline{\psi}}^n+\frac{(\Delta t)^3}{3!} \left((1-6\beta)\dddot{\underline{\psi}}^n+6\beta \, \dddot{\underline{\psi}}^{n+1} \right),\\
\dot{\underline{\psi}}^{n+1}=& \, \dot{\underline{\psi}}^n+\Delta t \ddot{\underline{\psi}}^n +\frac{(\Delta t)^2}{2!} \left((1-2\gamma)\dddot{\underline{\psi}}^n+2\gamma \, \dddot{\underline{\psi}}^{n+1} \right),\\
\ddot{\underline{\psi}}^{n+1}=&\, \ddot{\underline{\psi}}^{n}+\Delta t\, \left((1-\eta)\dddot{\underline{\psi}}^n+\eta \, \dddot{\underline{\psi}}^{n+1} \right). 
\end{aligned}
\end{equation}
 The average acceleration scheme corresponds to taking the Newmark parameters $(\beta, \gamma)=(1/4, 1/2)$ in \eqref{Newmark_West} for the Westervelt equation and $(\beta, \gamma, \eta)=(1/12, 1/4, 1/2)$ in \eqref{Newmark_JMGT} for the Jordan--Moore--Gibson--Thompson equation \eqref{WesterveltMC}, which is what we use in all the experiments. \\
\indent We set the initial conditions to
\begin{equation} \label{num_initial_data}
\begin{aligned}
    \left(\psi_0, \psi_1, \psi_2 \right)= \left (0, \, \mathcal{A}\, \text{exp}\left (-\frac{(x-0.1)^2}{2\sigma^2} \right),\, 0 \right),
    \end{aligned}
\end{equation}
with $\mathcal{A}=8\cdot 10^{4} \, \textup{m}^2/\textup{s}^2$ and $\sigma=0.01$,
meaning that we normalize potential (which is determined by $\vec{v}=-\nabla\psi$ only up to a constant) such that it vanishes at $t=0$,  drive the system by an inital pressure (based on the idetity $\rho \psi_t= p$) concentrated at $x=0.1$, and assume vanishing initial acceleration. Discretization in time is performed with $800$ time steps for the final time $\textup{T}=45 \, \mu$s. The spatial and temporal refinement always remain the same for both equations and different values of the relaxation time $\tau$. All the numerical results are obtained with the help of the GeoPDEs package in MATLAB\cite{Vazquez}.\\ 
\indent Figure~\ref{fig:snapshots} displays on the left side snapshots of the pressure wave $u=\varrho \psi_t$ obtained by employing equation \eqref{WesterveltMC} with the relaxation time set to $\tau=0.1 \,\mu$s. We observe the nonlinear steepening of the wave as it propagates. On the right, we see how the pressure profile changes with decreasing relaxation time. The pressure wave for $\tau =0\,\mu$s is computed by solving the Westervelt equation. \\
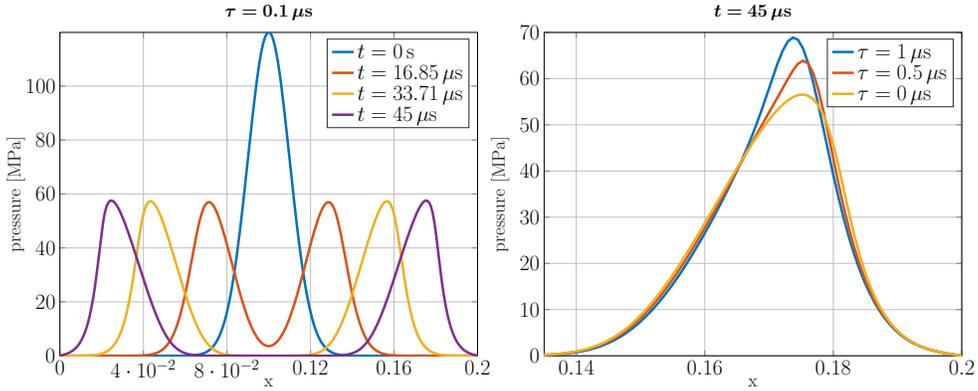
\begin{figure}[h!]
	\begin{center}
%
\definecolor{mycolor1}{rgb}{0.00000,0.44700,0.74100}%
\definecolor{mycolor2}{rgb}{0.85000,0.32500,0.09800}%
\definecolor{mycolor3}{rgb}{0.92900,0.69400,0.12500}%
\definecolor{mycolor4}{rgb}{0.49400,0.18400,0.55600}%
\begin{tikzpicture}[scale=0.475, font=\LARGE]

\begin{axis}[%
width=4.602in,
height=3.566in,
at={(0.772in,0.481in)},
scale only axis,
xmin=0,
xmax=0.200000000000003, xtick={0, 0.04, 0.08, 0.12, 0.16, 0.2},
xlabel style={font=\Large\color{white!15!black}},
xlabel={x},
ymin=0,
ymax=119.904864483392,
ylabel style={font=\Large\color{white!15!black}},
ylabel={pressure [MPa]},
axis background/.style={fill=white},
xmajorgrids,
ymajorgrids,
title style={font=\bfseries},
title={\Large \boldsymbol{$\tau=0.1\,\mu$}s},
legend style={legend cell align=left, align=left, draw=white!15!black}
]
\addplot [color=mycolor1, line width=2.0pt]
  table[row sep=crcr]{%
0	0\\
0.0454183266932233	4.07641344111198e-05\\
0.0478087649402426	0.000146042150319659\\
0.0494023904382459	0.000331252247406155\\
0.0501992031872476	0.00049415521488072\\
0.0509960159362493	0.000732505317756704\\
0.0517928286852651	0.00107894931748831\\
0.0525896414342668	0.0015791892679573\\
0.0533864541832685	0.00229673141635089\\
0.0541832669322702	0.0033191672445696\\
0.0549800796812718	0.00476640528437144\\
0.0557768924302735	0.00680135921086844\\
0.0565737051792894	0.00964369302455736\\
0.0573705179282911	0.0135873251203265\\
0.0581673306772927	0.0190224953136919\\
0.0589641434262944	0.0264632960474813\\
0.0597609561752961	0.0365816523041218\\
0.0605577689242978	0.0502487928785342\\
0.0613545816733136	0.0685852745870932\\
0.0621513944223153	0.0930205840058704\\
0.062948207171317	0.125363229415854\\
0.0637450199203187	0.167882028076178\\
0.0645418326693203	0.223398969408834\\
0.065338645418322	0.295393572692774\\
0.0661354581673237	0.38811804079468\\
0.0669322709163396	0.506721726915004\\
0.0677290836653412	0.657382474932604\\
0.0685258964143429	0.847441272364748\\
0.0693227091633446	1.08553538899044\\
0.0701195219123463	1.38172380148285\\
0.0709163346613479	1.74759728174664\\
0.0717131474103638	2.19636413120905\\
0.0725099601593655	2.74290127245175\\
0.0733067729083672	3.40375937900305\\
0.0741035856573689	4.19711006388694\\
0.0749003984063705	5.14262299572498\\
0.0756972111553722	6.26126130530864\\
0.0764940239043881	7.57498491171461\\
0.0772908366533898	9.10635353747759\\
0.0780876494023914	10.8780242619709\\
0.0788844621513931	12.9121424950669\\
0.0796812749003948	15.2296301903662\\
0.0804780876493965	17.8493808370798\\
0.0812749003984123	20.7873770718925\\
0.082071713147414	24.0557533579842\\
0.0828685258964157	27.6618327363655\\
0.0836653386454174	31.6071727538336\\
0.084462151394419	35.8866608621598\\
0.0860557768924366	45.3895542958401\\
0.0876494023904399	55.9692495642619\\
0.090039840637445	73.0734734485469\\
0.0924302788844642	90.1060088154789\\
0.0940239043824675	100.375917545742\\
0.0948207171314692	104.937649755919\\
0.0956175298804851	109.012360751649\\
0.0964143426294868	112.528559493477\\
0.0972111553784885	115.423004363733\\
0.0980079681274901	117.642592810675\\
0.0988047808764918	119.145980999281\\
0.0996015936254935	119.904864483392\\
0.100398406374495	119.904864483391\\
0.101195219123511	119.145980999281\\
0.101992031872513	117.642592810675\\
0.102788844621514	115.423004363733\\
0.103585657370516	112.528559493476\\
0.104382470119518	109.012360751649\\
0.105179282868519	104.937649755919\\
0.105976095617535	100.375917545742\\
0.107569721115539	90.1060088154785\\
0.109163346613542	78.8584899812607\\
0.112350597609563	55.9692495642613\\
0.113944223107566	45.3895542958393\\
0.115537848605584	35.886660862159\\
0.116334661354585	31.6071727538328\\
0.117131474103587	27.6618327363649\\
0.117928286852589	24.0557533579836\\
0.118725099601591	20.787377071892\\
0.119521912350592	17.8493808370794\\
0.120318725099608	15.2296301903658\\
0.12111553784861	12.9121424950665\\
0.121912350597611	10.8780242619706\\
0.122709163346613	9.1063535374773\\
0.123505976095615	7.57498491171437\\
0.124302788844616	6.26126130530841\\
0.125099601593632	5.14262299572479\\
0.125896414342634	4.19711006388678\\
0.126693227091636	3.40375937900291\\
0.127490039840637	2.74290127245162\\
0.128286852589639	2.19636413120898\\
0.129083665338641	1.74759728174656\\
0.129880478087657	1.38172380148278\\
0.130677290836658	1.08553538899039\\
0.13147410358566	0.847441272364705\\
0.132270916334662	0.657382474932561\\
0.133067729083663	0.506721726914975\\
0.133864541832665	0.388118040794666\\
0.134661354581667	0.29539357269276\\
0.135458167330682	0.223398969408819\\
0.136254980079684	0.167882028076164\\
0.137051792828686	0.125363229415839\\
0.137848605577688	0.0930205840058562\\
0.138645418326689	0.0685852745870932\\
0.139442231075691	0.0502487928785342\\
0.140239043824707	0.0365816523041218\\
0.141035856573708	0.0264632960474813\\
0.14183266932271	0.0190224953136919\\
0.142629482071712	0.0135873251203265\\
0.143426294820713	0.00964369302454315\\
0.144223107569715	0.00680135921086844\\
0.145019920318731	0.00476640528437144\\
0.145816733067733	0.0033191672445696\\
0.146613545816734	0.00229673141635089\\
0.147410358565736	0.0015791892679573\\
0.148207171314738	0.00107894931748831\\
0.149003984063739	0.000732505317756704\\
0.149800796812755	0.00049415521488072\\
0.151394422310759	0.000220646538110714\\
0.152988047808762	9.60510501784029e-05\\
0.156175298804783	1.68665110891197e-05\\
0.16494023904383	8.35561166923071e-08\\
0.200000000000003	0\\
};
\addlegendentry{$t=0\,$s}

\addplot [color=mycolor2, line width=2.0pt]
  table[row sep=crcr]{%
0	0\\
0.0215139442231091	4.16654684229911e-05\\
0.0239043824701213	0.000144594394384967\\
0.0254980079681246	0.000321086235892665\\
0.0262948207171334	0.000473939610465379\\
0.0270916334661351	0.000695134113129825\\
0.0278884462151368	0.00101311473937926\\
0.0286852589641455	0.00146721354433055\\
0.0294820717131472	0.00211141201524612\\
0.0302788844621489	0.00301924273475151\\
0.0310756972111577	0.00429011566851045\\
0.0318725099601593	0.00605740525212894\\
0.032669322709161	0.00849868754114169\\
0.0334661354581698	0.0118485696414936\\
0.0342629482071715	0.0164146031873997\\
0.0350597609561731	0.0225968155788863\\
0.0358565737051819	0.030911421910595\\
0.0366533864541836	0.0420192910843085\\
0.0374501992031853	0.056759724972153\\
0.038247011952194	0.0761900628916905\\
0.0390438247011957	0.101631538434212\\
0.0398406374501974	0.134721685870204\\
0.0406374501992062	0.17747341407847\\
0.0414342629482078	0.232340633842391\\
0.0422310756972095	0.302290037690327\\
0.0430278884462183	0.390878289899632\\
0.04382470119522	0.502333488209523\\
0.0446215139442216	0.641639307995241\\
0.0454183266932304	0.814619731496343\\
0.0462151394422321	1.02802169163029\\
0.0470119521912338	1.28959230670842\\
0.0478087649402426	1.60814662295174\\
0.0486055776892442	1.99362087630078\\
0.0494023904382459	2.45710517762201\\
0.0501992031872476	3.01084814240632\\
0.0509960159362564	3.6682242367851\\
0.051792828685258	4.44365239256931\\
0.0525896414342597	5.35245164930082\\
0.0533864541832685	6.41061612526057\\
0.0541832669322702	7.63448747951038\\
0.0549800796812718	9.04029832221742\\
0.0557768924302806	10.6435551571389\\
0.0565737051792823	12.4582252963073\\
0.057370517928284	14.4956905179164\\
0.0581673306772927	16.7634341208608\\
0.0589641434262944	19.2634424380246\\
0.0597609561752961	21.9903340968392\\
0.0613545816733065	28.0539493211294\\
0.0637450199203187	38.0726724798459\\
0.0653386454183291	44.5839639249041\\
0.0661354581673308	47.5320551369275\\
0.0669322709163325	50.1630088735027\\
0.0677290836653412	52.4090614075672\\
0.0685258964143429	54.2226648471762\\
0.0693227091633446	55.5788431062481\\
0.0701195219123534	56.4740804720148\\
0.0709163346613551	56.9225889654773\\
0.0717131474103567	56.9513286701101\\
0.0725099601593655	56.5950630544668\\
0.0733067729083672	55.8922630828406\\
0.0741035856573689	54.8821593930344\\
0.0749003984063776	53.6028814612297\\
0.0756972111553793	52.0904533256463\\
0.076494023904381	50.3783837233606\\
0.0772908366533898	48.4976260127284\\
0.0788844621513931	44.3421542833594\\
0.0804780876494036	39.8285215860812\\
0.0860557768924295	23.6154496543868\\
0.0876494023904399	19.4310354772038\\
0.0892430278884433	15.6432228983047\\
0.0900398406374521	13.9217319387409\\
0.0908366533864537	12.3247838483346\\
0.0916334661354554	10.8576478963382\\
0.0924302788844642	9.52409454554502\\
0.0932270916334659	8.32649095327954\\
0.0940239043824675	7.26594701993441\\
0.0948207171314763	6.34250822219214\\
0.095617529880478	5.55538708941687\\
0.0964143426294797	4.90322112296704\\
0.0972111553784885	4.38434167856968\\
0.0980079681274901	3.99703624776642\\
0.0988047808764918	3.73978598086067\\
0.0996015936255006	3.61146133226069\\
0.100398406374502	3.61146133224933\\
0.101195219123504	3.73978598086804\\
0.101992031872513	3.9970362477579\\
0.102788844621514	4.38434167857994\\
0.103585657370516	4.90322112295631\\
0.104382470119525	5.55538708942572\\
0.105179282868527	6.3425082221838\\
0.105976095617528	7.26594701994198\\
0.106772908366537	8.32649095327724\\
0.107569721115539	9.52409454554254\\
0.10836653386454	10.857647896338\\
0.109163346613549	12.3247838483263\\
0.109960159362551	13.9217319387489\\
0.111553784860561	17.4824432093486\\
0.113147410358565	21.4791373590944\\
0.114741035856575	25.8273228917877\\
0.117131474103587	32.7714957768728\\
0.120318725099601	42.1184276151609\\
0.121912350597611	46.4767415078724\\
0.123505976095615	50.378383723408\\
0.124302788844624	52.0904533255856\\
0.125099601593625	53.6028814612907\\
0.125896414342627	54.8821593929728\\
0.126693227091636	55.892263082887\\
0.127490039840637	56.5950630544274\\
0.128286852589639	56.9513286701345\\
0.129083665338648	56.922588965466\\
0.129880478087649	56.4740804720108\\
0.130677290836651	55.5788431062508\\
0.13147410358566	54.2226648471732\\
0.132270916334662	52.4090614075643\\
0.133067729083663	50.1630088735001\\
0.133864541832672	47.5320551369312\\
0.134661354581674	44.5839639248924\\
0.136254980079684	38.0726724798329\\
0.138645418326696	28.0539493211287\\
0.1402390438247	21.9903340968369\\
0.141035856573708	19.2634424380246\\
0.14183266932271	16.7634341208592\\
0.142629482071712	14.4956905179162\\
0.143426294820721	12.4582252963062\\
0.144223107569722	10.6435551571386\\
0.145019920318724	9.04029832221679\\
0.145816733067733	7.63448747951009\\
0.146613545816734	6.4106161252602\\
0.147410358565736	5.35245164930049\\
0.148207171314738	4.44365239256915\\
0.149003984063746	3.66822423678489\\
0.149800796812748	3.01084814240618\\
0.15059760956175	2.45710517762188\\
0.151394422310759	1.99362087630065\\
0.15219123505976	1.60814662295165\\
0.152988047808762	1.28959230670836\\
0.153784860557771	1.02802169163022\\
0.154581673306772	0.814619731496315\\
0.155378486055774	0.64163930799522\\
0.156175298804783	0.502333488209487\\
0.156972111553785	0.390878289899625\\
0.157768924302786	0.302290037690312\\
0.158565737051795	0.232340633842377\\
0.159362549800797	0.177473414078463\\
0.160159362549798	0.134721685870197\\
0.160956175298807	0.101631538434205\\
0.161752988047809	0.0761900628916834\\
0.16254980079681	0.056759724972153\\
0.163346613545819	0.0420192910843014\\
0.164143426294821	0.030911421910595\\
0.164940239043823	0.0225968155788863\\
0.165737051792831	0.0164146031873997\\
0.166533864541833	0.0118485696414936\\
0.167330677290835	0.00849868754114169\\
0.168127490039844	0.00605740525212894\\
0.168924302788845	0.00429011566851045\\
0.169721115537847	0.00301924273474441\\
0.170517928286856	0.00211141201524612\\
0.171314741035857	0.00146721354433055\\
0.172111553784859	0.00101311473937926\\
0.172908366533868	0.000695134113129825\\
0.173705179282869	0.000473939610465379\\
0.17529880478088	0.000216154624759213\\
0.176892430278883	9.61130610050986e-05\\
0.180079681274897	1.76094571884278e-05\\
0.188844621513944	9.80669767614017e-08\\
0.200000000000003	0\\
};
\addlegendentry{$t=16.85\, \mu$s}

\addplot [color=mycolor3, line width=2.0pt]
  table[row sep=crcr]{%
0	0\\
0.00159362549801045	0.000504941101105771\\
0.00239043824701213	0.000841757886909988\\
0.0031872509960138	0.00128992335789491\\
0.00398406374502258	0.0019042704781782\\
0.00478087649402426	0.00275650638612035\\
0.00557768924302593	0.00394161233762702\\
0.00637450199203471	0.0055858051238431\\
0.00717131474103638	0.00785647374986809\\
0.00796812749003806	0.0109745528543002\\
0.00876494023904684	0.0152298431836897\\
0.00956175298804851	0.0209998359908923\\
0.0103585657370502	0.0287726386641367\\
0.011155378486059	0.0391746291299597\\
0.0119521912350606	0.0530034826796921\\
0.0127490039840623	0.0712672134549237\\
0.0135458167330711	0.0952298516460033\\
0.0143426294820728	0.126464335939055\\
0.0151394422310744	0.166913140390818\\
0.0159362549800832	0.218957079903156\\
0.0167330677290849	0.285492654818171\\
0.0175298804780866	0.370018211451494\\
0.0183266932270882	0.476729119509649\\
0.019123505976097	0.610622106404911\\
0.0199203187250987	0.777608844801215\\
0.0207171314741004	0.984638853039321\\
0.0215139442231091	1.23983172110793\\
0.0223107569721108	1.5526185586613\\
0.0231075697211125	1.9338923046076\\
0.0239043824701213	2.39616598191267\\
0.024701195219123	2.9537368849629\\
0.0254980079681246	3.62285271008875\\
0.0262948207171334	4.42187203321191\\
0.0270916334661351	5.37140552503057\\
0.0278884462151368	6.49441391085804\\
0.0286852589641455	7.81622160415951\\
0.0294820717131472	9.36437713272077\\
0.0302788844621489	11.1682456073188\\
0.0310756972111577	13.2581486483036\\
0.0318725099601593	15.6637577718275\\
0.032669322709161	18.411298261787\\
0.0334661354581698	21.5189423104338\\
0.0342629482071715	24.9896348806582\\
0.0350597609561731	28.8007713248086\\
0.0366533864541836	37.1475427895464\\
0.038247011952194	45.4327578515066\\
0.0390438247011957	49.0255866624118\\
0.0398406374501974	52.0062962015357\\
0.0406374501992062	54.2912755986459\\
0.0414342629482078	55.8858146539029\\
0.0422310756972095	56.8537321703396\\
0.0430278884462183	57.2820811991423\\
0.04382470119522	57.2577572186508\\
0.0446215139442216	56.8573056704347\\
0.0454183266932304	56.1446523496197\\
0.0462151394422321	55.1721492317089\\
0.0470119521912338	53.9825540928493\\
0.0478087649402426	52.611006897269\\
0.0486055776892442	51.0867154985625\\
0.0501992031872476	47.6748882147469\\
0.051792828685258	43.9065239826024\\
0.0541832669322702	37.8531545257059\\
0.0589641434262944	25.4988032175225\\
0.0605577689243049	21.6130007083164\\
0.0621513944223082	17.9660224442819\\
0.0637450199203187	14.6199062928174\\
0.0653386454183291	11.6249274540749\\
0.0661354581673308	10.2707151812618\\
0.0669322709163325	9.01598958015686\\
0.0677290836653412	7.86219870079712\\
0.0685258964143429	6.80951559124103\\
0.0693227091633446	5.85683779199389\\
0.0701195219123534	5.00182880372881\\
0.0709163346613551	4.24100268838603\\
0.0717131474103567	3.56984952971134\\
0.0725099601593655	2.9829959953415\\
0.0733067729083672	2.47439218815398\\
0.0741035856573689	2.03751379427671\\
0.0749003984063776	1.66556754273192\\
0.0756972111553793	1.35168831051249\\
0.076494023904381	1.08911772425989\\
0.0772908366533898	0.871356536301725\\
0.0780876494023914	0.69228596735811\\
0.0788844621513931	0.546256182812591\\
0.0796812749004019	0.428142724009469\\
0.0804780876494036	0.333373801948447\\
0.0812749003984052	0.25793275642247\\
0.0820717131474069	0.19834069754566\\
0.0828685258964157	0.151624472244698\\
0.0836653386454174	0.115274779548812\\
0.084462151394419	0.0871986509268012\\
0.0852589641434278	0.0656697554442331\\
0.0860557768924295	0.0492791985176595\\
0.0868525896414312	0.0368887348614209\\
0.0876494023904399	0.0275876604690382\\
0.0884462151394416	0.0206541066378705\\
0.0892430278884433	0.0155210363689022\\
0.0900398406374521	0.011746931070995\\
0.0908366533864537	0.00899094033592007\\
0.0916334661354554	0.00699213080332584\\
0.0924302788844642	0.00555239674039143\\
0.0932270916334659	0.00452256684494046\\
0.0940239043824675	0.00379124774520534\\
0.0948207171314763	0.00327597158275239\\
0.095617529880478	0.00291625655376038\\
0.0964143426294797	0.00266823618770218\\
0.0972111553784885	0.00250056448334846\\
0.0988047808764918	0.00232594280308973\\
0.100398406374502	0.00229536382254736\\
0.101992031872513	0.00239135278206959\\
0.102788844621514	0.00250056446737545\\
0.103585657370516	0.00266823620290069\\
0.104382470119525	0.00291625654416805\\
0.105179282868527	0.00327597159135706\\
0.105976095617528	0.00379124773077422\\
0.106772908366537	0.00452256685561991\\
0.107569721115539	0.00555239672911512\\
0.10836653386454	0.00699213081387029\\
0.109163346613549	0.00899094032241265\\
0.109960159362551	0.0117469310801539\\
0.110756972111552	0.0155210363558638\\
0.111553784860561	0.0206541066463757\\
0.112350597609563	0.027587660457705\\
0.113147410358565	0.0368887348663165\\
0.113944223107573	0.0492791985065963\\
0.114741035856575	0.065669755447658\\
0.115537848605577	0.0871986509154468\\
0.116334661354578	0.115274779557943\\
0.117131474103587	0.151624472236115\\
0.117928286852589	0.198340697547494\\
0.118725099601591	0.257932756412799\\
0.119521912350599	0.33337380194704\\
0.120318725099601	0.428142724005497\\
0.121115537848603	0.54625618280717\\
0.121912350597611	0.692285967353648\\
0.122709163346613	0.871356536298244\\
0.123505976095615	1.08911772426188\\
0.124302788844624	1.35168831050018\\
0.125099601593625	1.66556754274121\\
0.125896414342627	2.03751379425613\\
0.126693227091636	2.47439218816302\\
0.127490039840637	2.9829959953232\\
0.128286852589639	3.56984952972986\\
0.129083665338648	4.24100268836261\\
0.129880478087649	5.00182880373862\\
0.130677290836651	5.8568377919727\\
0.13147410358566	6.80951559125509\\
0.132270916334662	7.86219870077541\\
0.133067729083663	9.01598958016321\\
0.133864541832672	10.2707151812507\\
0.134661354581674	11.6249274540743\\
0.136254980079684	14.6199062928176\\
0.137848605577688	17.9660224442707\\
0.139442231075698	21.6130007083029\\
0.141035856573708	25.4988032174485\\
0.143426294820721	31.6192153002457\\
0.147410358565736	41.9286213272112\\
0.149003984063746	45.8268540532299\\
0.15059760956175	49.4343078000017\\
0.151394422310759	51.0867154993161\\
0.15219123505976	52.6110068975898\\
0.152988047808762	53.9825540912545\\
0.153784860557771	55.1721492345637\\
0.154581673306772	56.1446523457267\\
0.155378486055774	56.8573056748679\\
0.156175298804783	57.2577572142794\\
0.156972111553785	57.2820812028462\\
0.157768924302786	56.8537321677517\\
0.158565737051795	55.8858146551553\\
0.159362549800797	54.291275598604\\
0.160159362549798	52.0062962007143\\
0.160956175298807	49.0255866636482\\
0.161752988047809	45.4327578502141\\
0.163346613545819	37.1475427887414\\
0.164940239043823	28.8007713245769\\
0.165737051792831	24.9896348806967\\
0.166533864541833	21.5189423105061\\
0.167330677290835	18.4112982616625\\
0.168127490039844	15.6637577719532\\
0.168924302788845	13.2581486482003\\
0.169721115537847	11.1682456073856\\
0.170517928286856	9.36437713268688\\
0.171314741035857	7.81622160416642\\
0.172111553784859	6.4944139108664\\
0.172908366533868	5.37140552501477\\
0.173705179282869	4.42187203322713\\
0.174501992031871	3.62285271007706\\
0.17529880478088	2.95373688496914\\
0.176095617529882	2.39616598191058\\
0.176892430278883	1.93389230460647\\
0.177689243027892	1.55261855866355\\
0.178486055776894	1.23983172110545\\
0.179282868525895	0.984638853041076\\
0.180079681274897	0.77760884480017\\
0.180876494023906	0.610622106405202\\
0.181673306772907	0.476729119509713\\
0.182470119521909	0.370018211451203\\
0.183266932270918	0.285492654818427\\
0.18406374501992	0.218957079902928\\
0.184860557768921	0.166913140390953\\
0.18565737051793	0.126464335938962\\
0.186454183266932	0.0952298516460388\\
0.187250996015933	0.0712672134549095\\
0.188047808764942	0.0530034826796708\\
0.188844621513944	0.0391746291299953\\
0.189641434262946	0.0287726386641012\\
0.190438247011954	0.0209998359909065\\
0.191235059760956	0.0152298431836755\\
0.192031872509958	0.0109745528543002\\
0.192828685258966	0.00785647374986809\\
0.193625498007968	0.0055858051238431\\
0.19442231075697	0.00394161233762702\\
0.195219123505979	0.00275650638612035\\
0.19601593625498	0.0019042704781782\\
0.196812749003982	0.00128992335789491\\
0.197609561752991	0.000841757886909988\\
0.198406374501992	0.000504941101105771\\
0.200000000000003	0\\
};
\addlegendentry{$t=33.71 \,\mu$s}

\addplot [color=mycolor4, line width=2.0pt]
  table[row sep=crcr]{%
0	0\\
0.00159362549801045	0.33298264725412\\
0.00239043824701213	0.520394744891092\\
0.0031872509960138	0.733557702714698\\
0.00398406374502258	0.982151501306404\\
0.00478087649402426	1.2768261374726\\
0.00557768924302593	1.62947347097414\\
0.00637450199203471	2.05352489519763\\
0.00717131474103638	2.56428386103934\\
0.00796812749003806	3.17928440373888\\
0.00876494023904684	3.91872991236013\\
0.00956175298804851	4.80595086005963\\
0.0103585657370502	5.86796290738161\\
0.011155378486059	7.13609974098096\\
0.0119521912350606	8.64656525830548\\
0.0127490039840623	10.441194961957\\
0.0135458167330711	12.5675983627159\\
0.0143426294820728	15.0789073697379\\
0.0151394422310744	18.03208866107\\
0.0159362549800832	21.4816368827865\\
0.0167330677290849	25.4687041605079\\
0.0175298804780866	29.9950262854585\\
0.0207171314741004	49.6651104021975\\
0.0215139442231091	53.0418389058298\\
0.0223107569721108	55.3405581231824\\
0.0231075697211125	56.7243953782377\\
0.0239043824701213	57.4049461238582\\
0.024701195219123	57.5543609492983\\
0.0254980079681246	57.2959069885151\\
0.0262948207171334	56.7174469495641\\
0.0270916334661351	55.8831868630186\\
0.0278884462151368	54.8412155191293\\
0.0286852589641455	53.6284107520308\\
0.0294820717131472	52.2737358260192\\
0.0310756972111577	49.2277334835927\\
0.032669322709161	45.8459459206904\\
0.0342629482071715	42.2330446124712\\
0.0366533864541836	36.5593045757913\\
0.0414342629482078	25.1263823129389\\
0.0430278884462183	21.5128364512382\\
0.0446215139442216	18.0930269192416\\
0.0462151394422321	14.9190507256151\\
0.0478087649402426	12.0367649196052\\
0.0494023904382459	9.48262602752452\\
0.0501992031872476	8.33659184419309\\
0.0509960159362564	7.28029594946852\\
0.051792828685258	6.31422318316596\\
0.0525896414342597	5.43777214718324\\
0.0533864541832685	4.64925792365801\\
0.0541832669322702	3.9459553974859\\
0.0549800796812718	3.32418222049368\\
0.0557768924302806	2.779420156176\\
0.0565737051792823	2.30646705240495\\
0.057370517928284	1.89961155089741\\
0.0581673306772927	1.55281818893248\\
0.0589641434262944	1.25991219337351\\
0.0597609561752961	1.0147520403853\\
0.0605577689243049	0.811381507687209\\
0.0613545816733065	0.644154196037604\\
0.0621513944223082	0.507827789936144\\
0.062948207171317	0.397627176718807\\
0.0637450199203187	0.309278830394561\\
0.0645418326693203	0.239019867739721\\
0.0653386454183291	0.183586762288272\\
0.0661354581673308	0.140188547992068\\
0.0669322709163325	0.106469531649878\\
0.0677290836653412	0.0804656800451227\\
0.0685258964143429	0.0605583494972208\\
0.0693227091633446	0.0454280387315364\\
0.0701195219123534	0.0340102241837599\\
0.0709163346613551	0.0254545504721193\\
0.0717131474103567	0.0190881729301253\\
0.0725099601593655	0.0143835524882547\\
0.0733067729083672	0.0109307285760636\\
0.0741035856573689	0.00841384432469283\\
0.0749003984063776	0.00659158327875531\\
0.0756972111553793	0.00528108785481152\\
0.076494023904381	0.00434491203294129\\
0.0772908366533898	0.00368056094323066\\
0.0780876494023914	0.00321219940898487\\
0.0788844621513931	0.00288415007331366\\
0.0796812749004019	0.00265584677543984\\
0.0812749003984052	0.00238942810927512\\
0.0828685258964157	0.00226486534356951\\
0.0860557768924295	0.00218212888929514\\
0.0948207171314763	0.00216038204390401\\
0.115537848605577	0.00220784014751985\\
0.117928286852589	0.00231526544030203\\
0.119521912350599	0.00249795662499963\\
0.120318725099601	0.00265584676881758\\
0.121115537848603	0.00288415007207021\\
0.121912350597611	0.00321219940585848\\
0.122709163346613	0.00368056094111324\\
0.123505976095615	0.00434491202119602\\
0.124302788844624	0.00528108785846371\\
0.125099601593625	0.00659158327038512\\
0.125896414342627	0.00841384432606418\\
0.126693227091636	0.0109307285634159\\
0.127490039840637	0.0143835524928519\\
0.128286852589639	0.0190881729138397\\
0.129083665338648	0.0254545504755654\\
0.129880478087649	0.0340102241665221\\
0.130677290836651	0.0454280387389403\\
0.13147410358566	0.0605583494816102\\
0.132270916334662	0.0804656800446253\\
0.133067729083663	0.106469531640947\\
0.133864541832672	0.140188547999756\\
0.134661354581674	0.183586762276484\\
0.135458167330675	0.239019867734797\\
0.136254980079684	0.309278830386553\\
0.137051792828686	0.397627176716718\\
0.137848605577688	0.507827789922317\\
0.138645418326696	0.644154196036425\\
0.139442231075698	0.811381507678206\\
0.1402390438247	1.01475204038128\\
0.141035856573708	1.25991219337089\\
0.14183266932271	1.55281818892778\\
0.142629482071712	1.8996115508825\\
0.143426294820721	2.30646705240122\\
0.144223107569722	2.77942015616188\\
0.145019920318724	3.32418222050867\\
0.145816733067733	3.94595539745991\\
0.146613545816734	4.64925792366983\\
0.147410358565736	5.43777214714849\\
0.148207171314738	6.31422318319171\\
0.149003984063746	7.28029594943737\\
0.149800796812748	8.33659184419407\\
0.15059760956175	9.48262602754167\\
0.15219123505976	12.0367649197854\\
0.153784860557771	14.919050726166\\
0.155378486055774	18.093026920515\\
0.156972111553785	21.5128364537926\\
0.159362549800797	26.9884403592745\\
0.163346613545819	36.5593045867178\\
0.165737051792831	42.2330446072858\\
0.167330677290835	45.8459459286385\\
0.168924302788845	49.2277335133265\\
0.170517928286856	52.2737358809374\\
0.171314741035857	53.6284106869747\\
0.172111553784859	54.841215590079\\
0.172908366533868	55.8831867922787\\
0.173705179282869	56.7174470127662\\
0.174501992031871	57.2959069399986\\
0.17529880478088	57.5543609776804\\
0.176095617529882	57.4049461178037\\
0.176892430278883	56.7243953640201\\
0.177689243027892	55.3405581515031\\
0.178486055776894	53.0418388718653\\
0.179282868525895	49.6651104337544\\
0.180079681274897	45.2656101109934\\
0.182470119521909	29.9950262860884\\
0.183266932270918	25.4687041631649\\
0.18406374501992	21.481636878839\\
0.184860557768921	18.032088664991\\
0.18565737051793	15.0789073665238\\
0.186454183266932	12.5675983649949\\
0.187250996015933	10.4411949605562\\
0.188047808764942	8.64656525901333\\
0.188844621513944	7.13609974074588\\
0.189641434262946	5.86796290733648\\
0.190438247011954	4.80595086023902\\
0.191235059760956	3.91872991213954\\
0.192031872509958	3.17928440394834\\
0.192828685258966	2.56428386086345\\
0.193625498007968	2.05352489533386\\
0.19442231075697	1.62947347087449\\
0.195219123505979	1.27682613754155\\
0.19601593625498	0.982151501261342\\
0.196812749003982	0.733557702741969\\
0.197609561752991	0.520394744876072\\
0.198406374501992	0.332982647261325\\
0.200000000000003	0\\
};
\addlegendentry{$t=45\, \mu$s}

\end{axis}
\end{tikzpicture}%
\hspace*{-1.8mm} 
%
\definecolor{mycolor1}{rgb}{0.00000,0.44700,0.74100}%
\definecolor{mycolor2}{rgb}{0.85000,0.32500,0.09800}%
\definecolor{mycolor3}{rgb}{0.92900,0.69400,0.12500}%
\begin{tikzpicture}[scale=0.475, font=\LARGE]

\begin{axis}[%
width=4.602in,
height=3.566in,
at={(0.772in,0.481in)},
scale only axis,
xmin=0.135,
xmax=0.2, xtick={0.14, 0.16, 0.18, 0.2},
xlabel style={font=\Large\color{white!15!black}},
xlabel={x},
ymin=0,
ymax=70,
ylabel style={font=\Large\color{white!15!black}},
ylabel={pressure [MPa]},
axis background/.style={fill=white},
xmajorgrids,
ymajorgrids,
title style={font=\bfseries},
title={\Large \boldsymbol{$t=45\,\mu$}s},
legend style={legend cell align=left, align=left, draw=white!15!black}
]
\addplot [color=mycolor1, line width=2.0pt]
  table[row sep=crcr]{%
0.134661354581667	0.146586064296613\\
0.135458167330682	0.18835560284225\\
0.136254980079684	0.242349627411315\\
0.137051792828686	0.311432472414978\\
0.137848605577688	0.398920787109162\\
0.138645418326689	0.508593717210971\\
0.139442231075691	0.644685934062437\\
0.140239043824707	0.811860494810134\\
0.141035856573708	1.01515953797782\\
0.14183266932271	1.25993223775913\\
0.142629482071712	1.5517411474158\\
0.143426294820713	1.89624988315489\\
0.144223107569715	2.29909681792516\\
0.145019920318731	2.76576084556142\\
0.145816733067733	3.30142614789561\\
0.146613545816734	3.91085312927024\\
0.147410358565736	4.59826224459182\\
0.148207171314738	5.36723640637396\\
0.149003984063739	6.22064616190237\\
0.149800796812755	7.1606000837083\\
0.150597609561757	8.18842102956577\\
0.151394422310759	9.30464729782888\\
0.15219123505976	10.509056382824\\
0.152988047808762	11.8007081256593\\
0.153784860557764	13.178003616227\\
0.15458167330678	14.6387562687096\\
0.155378486055781	16.1802721140734\\
0.156175298804783	17.7994376324978\\
0.156972111553785	19.4928155954419\\
0.157768924302786	21.2567527581558\\
0.158565737051788	23.0875083621872\\
0.159362549800804	24.9814198966656\\
0.160159362549805	26.9351329062923\\
0.160956175298807	28.9459345675065\\
0.161752988047809	31.0122441218985\\
0.16254980079681	33.1343209903226\\
0.163346613545812	35.3152405671667\\
0.164143426294814	37.5621350141045\\
0.16494023904383	39.8875655838513\\
0.165737051792831	42.3106369998091\\
0.166533864541833	44.8570395641551\\
0.167330677290835	47.5566110479219\\
0.168127490039836	50.4363778078231\\
0.168924302788838	53.5067681848704\\
0.171314741035857	63.2204387767926\\
0.172111553784859	65.9933232286102\\
0.172908366533861	68.0015766726021\\
0.173705179282862	68.893155860443\\
0.174501992031878	68.419937878696\\
0.17529880478088	66.5158529323397\\
0.176095617529882	63.3138543458718\\
0.176892430278883	59.0969952321079\\
0.177689243027885	54.2154532795212\\
0.180079681274904	38.6808501011064\\
0.180876494023906	33.9012765279211\\
0.181673306772907	29.5009416164829\\
0.182470119521909	25.5142506990459\\
0.183266932270911	21.9463266546376\\
0.184063745019927	18.7836996493853\\
0.184860557768928	16.0019971336055\\
0.18565737051793	13.571177047278\\
0.186454183266932	11.4589351493213\\
0.187250996015933	9.6328330770075\\
0.188047808764935	8.06156499757888\\
0.188844621513951	6.71566118165086\\
0.189641434262953	5.56783251743769\\
0.190438247011954	4.59309137879279\\
0.191235059760956	3.76873664786928\\
0.192031872509958	3.07425860348658\\
0.192828685258959	2.49119821906523\\
0.193625498007975	2.00298168558747\\
0.194422310756977	1.59474222641809\\
0.195219123505979	1.2531358172231\\
0.19601593625498	0.966154128970913\\
0.196812749003982	0.722936112308545\\
0.197609561752984	0.513578639756005\\
0.198406374501985	0.328946116757521\\
0.199203187251001	0.160478915635551\\
0.200000000000003	0\\
};
\addlegendentry{$\tau=1 \,\mu$s}

\addplot [color=mycolor2, line width=2.0pt]
  table[row sep=crcr]{%
0.134661354581674	0.172856720498245\\
0.135458167330675	0.223355167723525\\
0.136254980079684	0.287689754065852\\
0.137051792828686	0.3689636538288\\
0.137848605577688	0.470758301143789\\
0.138645418326696	0.597142930632728\\
0.139442231075698	0.752664603556205\\
0.1402390438247	0.942314824253259\\
0.141035856573708	1.17146998119484\\
0.14183266932271	1.44580456399468\\
0.142629482071712	1.77117829076823\\
0.143426294820721	2.15350071562419\\
0.144223107569722	2.59857924885982\\
0.145019920318724	3.11195848009847\\
0.145816733067733	3.69875991456601\\
0.146613545816734	4.36353152591224\\
0.147410358565736	5.11011581080859\\
0.148207171314738	5.94154343588261\\
0.149003984063746	6.85995731119965\\
0.149800796812748	7.8665693690054\\
0.15059760956175	8.96164978156812\\
0.151394422310759	10.1445461438775\\
0.15219123505976	11.4137284488875\\
0.152988047808762	12.7668546078803\\
0.153784860557771	14.200850802321\\
0.154581673306772	15.712000998861\\
0.155378486055774	17.2960404342034\\
0.156175298804783	18.9482485529741\\
0.156972111553785	20.6635377712202\\
0.157768924302786	22.4365352119243\\
0.158565737051795	24.2616554460784\\
0.160159362549798	28.04522222366\\
0.161752988047809	31.967381431434\\
0.164143426294821	38.0063718831429\\
0.167330677290835	46.0924321370293\\
0.168924302788845	50.043451084846\\
0.172111553784859	57.8657083391596\\
0.172908366533868	59.8573535719618\\
0.173705179282869	61.7338451492673\\
0.174501992031871	63.2167883809034\\
0.17529880478088	63.8796671548211\\
0.176095617529882	63.2752748463226\\
0.176892430278883	61.1444505149992\\
0.177689243027892	57.558060013679\\
0.178486055776894	52.8776575157648\\
0.179282868525895	47.5863886139004\\
0.180079681274897	42.1282604912735\\
0.180876494023906	36.8291237350917\\
0.181673306772907	31.8886035480154\\
0.182470119521909	27.4062720372048\\
0.183266932270918	23.4143285651814\\
0.18406374501992	19.9045049511637\\
0.184860557768921	16.8467698261965\\
0.18565737051793	14.2010649562003\\
0.186454183266932	11.9241766938873\\
0.187250996015933	9.97351418933079\\
0.188047808764942	8.309018738396\\
0.188844621513944	6.89400263921642\\
0.189641434262946	5.69539155775256\\
0.190438247011954	4.6836518673806\\
0.191235059760956	3.8325633285635\\
0.192031872509958	3.11892468016689\\
0.192828685258966	2.52223996131303\\
0.193625498007968	2.02440872494373\\
0.19442231075697	1.6094305549778\\
0.195219123505979	1.26312707144896\\
0.19601593625498	0.972880823870206\\
0.196812749003982	0.727388872493371\\
0.197609561752991	0.516428161756579\\
0.198406374501992	0.330629727637863\\
0.199203187250994	0.161259111491646\\
0.200000000000003	0\\
};
\addlegendentry{$\tau=0.5 \, \mu$s}

\addplot [color=mycolor3, line width=2.0pt]
  table[row sep=crcr]{%
0.134661354581674	0.185274706211494\\
0.135458167330675	0.241671214320974\\
0.136254980079684	0.313097443423885\\
0.137051792828686	0.402852550046418\\
0.137848605577688	0.514740283529378\\
0.138645418326696	0.65308097207491\\
0.139442231075698	0.822702938296324\\
0.1402390438247	1.02890842920625\\
0.141035856573708	1.27741009244865\\
0.14183266932271	1.57423568968168\\
0.142629482071712	1.92560119242938\\
0.143426294820721	2.33775534839277\\
0.144223107569722	2.81680204915383\\
0.145019920318724	3.36850971614942\\
0.145816733067733	3.99811915658533\\
0.146613545816734	4.71016222353825\\
0.147410358565736	5.50830320326822\\
0.148207171314738	6.39521281544188\\
0.149003984063746	7.3724817485633\\
0.149800796812748	8.44057680714878\\
0.15059760956175	9.59883910413592\\
0.151394422310759	10.8455202717373\\
0.15219123505976	12.177850342005\\
0.152988047808762	13.5921293358257\\
0.153784860557771	15.0838342674804\\
0.154581673306772	16.647733500972\\
0.155378486055774	18.278001409515\\
0.156175298804783	19.9683275129115\\
0.156972111553785	21.7120155677189\\
0.158565737051795	25.3312639868525\\
0.160159362549798	29.0773331739284\\
0.164940239043823	40.4575127838447\\
0.165737051792831	42.2884887113909\\
0.166533864541833	44.0765281318376\\
0.167330677290835	45.8112142604975\\
0.168127490039844	47.4811192200558\\
0.168924302788845	49.073487874543\\
0.169721115537847	50.5738250983595\\
0.170517928286856	51.9653540110852\\
0.171314741035857	53.2282974721371\\
0.172111553784859	54.338916464448\\
0.172908366533868	55.2682154606617\\
0.173705179282869	55.9801913815294\\
0.174501992031871	56.4294829245964\\
0.17529880478088	56.5582814483026\\
0.176095617529882	56.2925289542742\\
0.176892430278883	55.5380034863114\\
0.177689243027892	54.1786741197141\\
0.178486055776894	52.0840737655347\\
0.179282868525895	49.139662291263\\
0.180079681274897	45.3144103509371\\
0.180876494023906	40.747270826429\\
0.182470119521909	30.8028948126736\\
0.183266932270918	26.171242643418\\
0.18406374501992	22.0427412307828\\
0.184860557768921	18.459076314383\\
0.18565737051793	15.3953896932982\\
0.186454183266932	12.7982119175145\\
0.187250996015933	10.6074432427969\\
0.188047808764942	8.76552185382271\\
0.188844621513944	7.22063255568701\\
0.189641434262946	5.92771222786189\\
0.190438247011954	4.84793576453214\\
0.191235059760956	3.94806181674937\\
0.192031872509958	3.19966000721233\\
0.192828685258966	2.57834591379026\\
0.193625498007968	2.06316901848672\\
0.19442231075697	1.6360397681828\\
0.195219123505979	1.28125985747023\\
0.19601593625498	0.985112934947686\\
0.196812749003982	0.735500955173634\\
0.197609561752991	0.521627740863323\\
0.198406374501992	0.333705538079421\\
0.199203187250994	0.162685480109594\\
0.200000000000003	0\\
};
\addlegendentry{$\tau=0 \,\mu$s}

\end{axis}
\end{tikzpicture}%
\caption{\textbf{(left)} Snapshots of the pressure $u=\varrho \psi_t$ for a fixed relaxation time $\tau=0.1 \, \mu$s \textbf{(right)} Pressure wave for different relaxation parameters $\tau$ at final time. 			\label{fig:snapshots}}
	\end{center}
\end{figure}
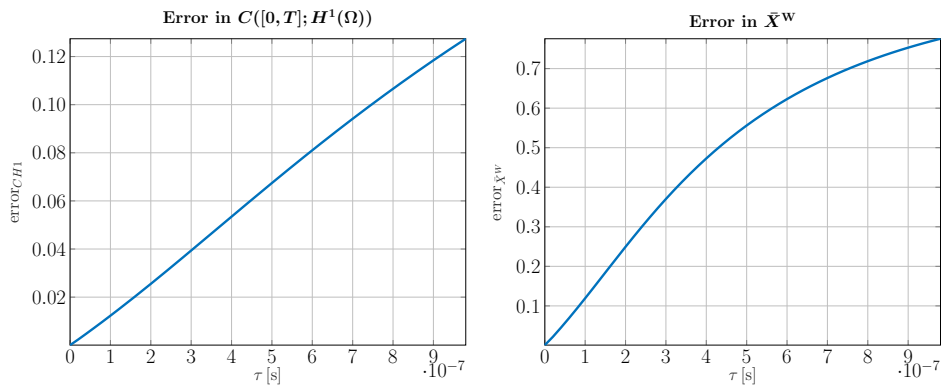
\begin{figure}[h!]
	\begin{center}
%
\definecolor{mycolor1}{rgb}{0.00000,0.44700,0.74100}%
\begin{tikzpicture}[scale=0.45, font=\LARGE]

\begin{axis}[%
width=4.602in,
height=3.566in,
at={(0.772in,0.481in)},
scale only axis,
xmin=1.00000008274037e-10,
xmax=9.80099999992046e-07,
xlabel style={font=\Large\color{white!15!black}},
xlabel={\Large $\tau \, [\textup{s}]$},
ymin=6.42731803469632e-05,
ymax=0.127400486593906,
ylabel style={at={(-0.03,0.5)}, font=\Large\color{white!15!black}},
ylabel={$\textup{error}_{CH1}$},
axis background/.style={fill=white},
y tick label style={
	/pgf/number format/.cd,
	fixed,
	fixed zerofill,
	precision=2,
	/tikz/.cd
},
xmajorgrids,
ymajorgrids,
title style={font=\bfseries},
title={\Large Error in \boldsymbol{$C([0,T]; H^1(\Omega))$}},
legend style={legend cell align=left, align=left, draw=white!15!black}
]
\addplot [color=mycolor1, line width=2.0pt]
  table[row sep=crcr]{%
9.80099999992046e-07	0.127400486593906\\
9.401000000131e-07	0.12296819671296\\
9.00100000006399e-07	0.118429670668816\\
8.60099999999697e-07	0.113784468687131\\
8.20099999992996e-07	0.109032646955549\\
7.80099999986295e-07	0.104174862535778\\
7.40100000007349e-07	0.0992124962822906\\
7.00100000000647e-07	0.094147796131937\\
6.60099999993946e-07	0.0889840429539007\\
6.20099999987245e-07	0.0837257406294813\\
5.80100000008299e-07	0.0783788308945823\\
5.40100000001598e-07	0.0729509313431745\\
5.00099999994896e-07	0.0674515912854545\\
4.40099999998722e-07	0.0590949804980019\\
3.60100000013075e-07	0.0478333985433927\\
2.80099999999672e-07	0.0365809209297603\\
2.40099999992971e-07	0.0310175442428655\\
2.0009999998627e-07	0.0255306384491289\\
1.60100000007324e-07	0.0201455685299709\\
1.20100000000622e-07	0.014884905004411\\
1.0010000001115e-07	0.0123070054582203\\
8.00999999939211e-08	0.0097664017274092\\
6.01000000044483e-08	0.00726435135496456\\
4.00999999872198e-08	0.00480174410407444\\
2.00999999977469e-08	0.00237927898201032\\
1.00000008274037e-10	6.42731803469632e-05\\
};
\end{axis}
\end{tikzpicture}%
\hspace*{2mm} 
%
\definecolor{mycolor1}{rgb}{0.00000,0.44700,0.74100}%
\begin{tikzpicture}[scale=0.45, font=\LARGE]

\begin{axis}[%
width=4.602in,
height=3.566in,
at={(0.772in,0.481in)},
scale only axis,
xmin=1.00000008274037e-10,
xmax=9.80100000047557e-07,
xlabel style={font=\Large\color{white!15!black}},
xlabel={\Large $\tau \, [\textup{s}]$},
ymin=0.000850673042971595,
ymax=0.775506875805468,
ylabel style={at={(-0.01,0.5)}, font=\Large\color{white!15!black}},
ylabel={$\textup{error}_{\bar{X}^W}$},
axis background/.style={fill=white},
xmajorgrids,
ymajorgrids,
title style={font=\bfseries},
title={\Large Error in \boldsymbol{$\bar{X}^{\textup{W}}$}},
legend style={legend cell align=left, align=left, draw=white!15!black}
]
\addplot [color=mycolor1, line width=2.0pt]
  table[row sep=crcr]{%
9.80100000047557e-07	0.775506875805468\\
9.60099999947062e-07	0.770235030351333\\
9.40099999957589e-07	0.764734330988659\\
9.20099999968116e-07	0.758993809393763\\
9.00099999978643e-07	0.753001914304792\\
8.8009999998917e-07	0.746746481742432\\
8.60099999999697e-07	0.740214704568574\\
8.40100000010224e-07	0.733393101622989\\
8.20100000020751e-07	0.726267486737512\\
8.00100000031279e-07	0.71882293802669\\
7.80100000041806e-07	0.711043767962051\\
7.60100000052333e-07	0.702913494879316\\
7.40099999951838e-07	0.694414816738766\\
7.20099999962365e-07	0.685529588172628\\
7.00099999972892e-07	0.676238802111098\\
6.80099999983419e-07	0.666522577603831\\
6.60099999993946e-07	0.656360155833154\\
6.40100000004473e-07	0.645729906808815\\
6.20100000015e-07	0.634609349799713\\
6.00100000025527e-07	0.622975191288366\\
5.80100000036055e-07	0.610803385073181\\
5.60100000046582e-07	0.598069220210226\\
5.40099999946086e-07	0.584747443727482\\
5.20099999956614e-07	0.570812426556258\\
5.00099999967141e-07	0.556238382915484\\
4.80099999977668e-07	0.540999655459236\\
4.60099999988195e-07	0.525071081047394\\
4.40099999998722e-07	0.508428454643195\\
4.20100000009249e-07	0.491049111824449\\
4.00100000019776e-07	0.472912653698107\\
3.80100000030303e-07	0.454001841376463\\
3.60100000040831e-07	0.434303689049905\\
3.40100000051358e-07	0.41381078617108\\
3.20099999950862e-07	0.392522877496462\\
3.0009999996139e-07	0.370448723297579\\
2.80099999971917e-07	0.347608247100277\\
2.60099999982444e-07	0.324034951097278\\
2.40099999992971e-07	0.299778532403848\\
2.20100000003498e-07	0.274907558638198\\
2.00100000014025e-07	0.249511950035212\\
1.60100000035079e-07	0.197623352866078\\
1.20099999945111e-07	0.145293964180653\\
1.00099999955638e-07	0.119412364826836\\
8.00999999661656e-08	0.0939694120943644\\
6.00999999766927e-08	0.0691374911532111\\
4.00999999872198e-08	0.0450620089195959\\
2.00999999977469e-08	0.0218555442874305\\
1.00000008274037e-10	0.000850673042971595\\
};

\end{axis}
\end{tikzpicture}%
\caption{Relative errors for varying relaxation time in \textbf{(left)} $C([0,T]; H^1(\Omega))$ and \textbf{(right)} $\bar{X}^{W}$.
			\label{fig:error}}
	\end{center}
\end{figure}\\
\indent To further illustrate the results from Section~\ref{sec:limits}, we solve equation \eqref{WesterveltMC} with the relaxation time varying over $\tau \in [10^{-4}, 1]\, \mu s$ and compute the difference to the solution of the Westervelt equation \eqref{Westervelt}.  We plot the relative errors in the $\bar{X}^W$ norm, defined in \eqref{XWXKnorms}, and in the $C([0,T]; H^1(\Omega))$ norm:
\begin{equation*}
\begin{aligned}
&  \textup{error}_{\bar{X}^\textup{W}}(\tau)=\dfrac{\|\psi^{\tau}-\bar{\psi}\|_{\bar{X}^{W}}}{\|\bar{\psi}\|_{\bar{X}^{W}}},\quad  \textup{error}_{CH^1}(\tau)=\dfrac{\|\psi^{\tau}-\bar{\psi}\|_{CH^1}}{\|\bar{\psi}\|_{CH^1}};
\end{aligned}
\end{equation*}
see Figure~\ref{fig:error}. The numerical errors decrease with the parameter $\tau$, in agreement with the theoretical results of Theorem~\ref{th:limits}. Figure~\ref{fig:error} even indicates a stronger result, i.e., strong convergence in the $\bar{X}^{\textup{W}}$ norm. For $\tau=10^{-10}\,$s, the errors amount to $\textup{error}_{CH^1}(\tau) \approx 6.43 \cdot 10^{-5}$ and $\textup{error}_{\bar{X}^{W}}(\tau) \approx 8.5 \cdot 10^{-4}$. The error plots also suggest a lower rate of convergence with respect to $\tau$ in the $\bar{X}^{W}$ norm.
\section*{Acknowledgments}
The second author acknowledges the funding provided by the Deutsche Forschungsgemeinschaft under the grant number WO 671/11-1.

\end{document}